\documentclass[12pt]{amsart}
\usepackage{amsmath}
\usepackage{amsfonts}
\usepackage{amssymb}
\usepackage{amscd}
\usepackage{mathtools}
\usepackage{amsxtra}
\usepackage{amsthm}
\usepackage{graphicx}
\usepackage[abbrev,alphabetic]{amsrefs}
\PassOptionsToPackage{dvipsnames}{xcolor}
\usepackage{soul,xcolor}
\setstcolor{red}
\usepackage{stmaryrd}
\usepackage{booktabs}
\usepackage{multirow}
\newtagform{tiny}{\tiny(}{)}

\usepackage{mathtools}
\usepackage{hyperref}
\usepackage[margin=1.25in]{geometry}
\usepackage{mathrsfs}

\usepackage{amsthm}
\usepackage{comment}
\usepackage[all,cmtip]{xy}
\usepackage{tikz-cd}
\usetikzlibrary{cd}

\tikzcdset{
  cells={font=\everymath\expandafter{\the\everymath\displaystyle}},
}

\usepackage[all]{xy}

\makeatletter
\def\@tocline#1#2#3#4#5#6#7{\relax
  \ifnum #1>\c@tocdepth 
  \else
    \par \addpenalty\@secpenalty\addvspace{#2}%
    \begingroup \hyphenpenalty\@M
    \@ifempty{#4}{%
      \@tempdima\csname r@tocindent\number#1\endcsname\relax
    }{%
      \@tempdima#4\relax
    }%
    \parindent\z@ \leftskip#3\relax \advance\leftskip\@tempdima\relax
    \rightskip\@pnumwidth plus4em \parfillskip-\@pnumwidth
    #5\leavevmode\hskip-\@tempdima
      \ifcase #1
       \or\or \hskip 1em \or \hskip 2em \else \hskip 3em \fi%
      #6\nobreak\relax
    \hfill\hbox to\@pnumwidth{\@tocpagenum{#7}}\par
    \nobreak
    \endgroup
  \fi}
\makeatother

\renewcommand{\mod}{\ \textrm{mod}\ }

\newcommand{\Z}{\mathbb{Z}}
\newcommand{\Q}{\mathbb{Q}}

\newcommand{\F}{\mathbb{F}}

\newcommand{\m}{\mathfrak{m}}
\newcommand{\n}{\mathfrak{n}}

\newcommand{\perf}{\mathrm{perf}}

\newcommand{\Proj}{\mathrm{Proj}}
\newcommand{\wt}{\widetilde}

\def\var{\overline}

\DeclareMathOperator{\Spec}{Spec}

\DeclareMathOperator{\Hom}{Hom}

\DeclareMathOperator{\Ker}{Ker}
\DeclareMathOperator{\reg}{reg}

\DeclareMathOperator{\fpt}{fpt}

\DeclareMathOperator{\ppt}{ppt}

\renewcommand{\Im}{\mathrm{Im}}

\newcommand{\sQ}[2]{\ensuremath{Q_{#1,\left(#2\right)}}}
\newcommand{\sPhi}[2]{\ensuremath{\Phi_{#1,\left(#2\right)}}}

\newtheorem{theoremA}{Theorem}

\theoremstyle{remark}
\newtheorem*{ackn}{Acknowledgements}

\theoremstyle{plain}

\newcommand{\bs}{\boldsymbol{s}}

\numberwithin{equation}{section}

  \newsavebox{\pullbackdl}
\sbox\pullbackdl{%
\begin{tikzpicture}%
\draw (-1ex,0ex) -- (0ex,0ex);%
\draw (0ex,-1ex) -- (0ex,0ex);%
\end{tikzpicture}}

\newsavebox{\pushoutdr}
\sbox\pushoutdr{%
\begin{tikzpicture}%
\draw (-1ex,-1ex) -- (-1ex,0ex);%
\draw (-1ex,0ex) -- (0ex,0ex);%
\end{tikzpicture}}

\usepackage[capitalise,noabbrev]{cleveref}
\usepackage{thmtools}

\declaretheorem[name=Theorem,numberwithin=section]{theorem}
\declaretheorem[sibling=theorem,name=Lemma]{lemma}
\declaretheorem[sibling=theorem,name=Proposition]{proposition}
\declaretheorem[sibling=theorem,name=Corollary]{corollary}

\declaretheorem[sibling=theorem,style=definition,name=Definition]{definition}
\declaretheorem[sibling=theorem,style=definition,name=Example]{example}
\declaretheorem[sibling=theorem,style=definition,name=Notation]{notation}
\declaretheorem[sibling=theorem,style=remark,name=Remark]{remark}
\declaretheorem[sibling=theorem,style=definition,name=Question]{question}



\title[A Criterion for Perfectoid Purity and the Rationality of Thresholds
]{A Criterion for Perfectoid Purity and the Rationality of Thresholds
}
\author{Shou Yoshikawa}
\address{Institute of Science Tokyo, Tokyo 152-8551, Japan}
\email{yoshikawa.s.9fe9@m.isct.ac.jp}

\begin{document}

\begin{abstract}
We introduce a new criterion providing a sufficient condition for a hypersurface in an unramified regular local ring to be \emph{perfectoid pure}.  
The criterion is formulated in terms of an explicitly computable sequence of integers, called the \emph{splitting-order sequence}.  
Our main theorem shows that if all entries of the sequence are at most $p-1$, then the hypersurface is perfectoid pure, and the perfectoid pure threshold can be computed explicitly from it.  
As a consequence, we prove that for any regular local ring $R$, the perfectoid pure threshold $\ppt(R,p)$ with respect to $p$ is always a rational number. 
Moreover, we show that for sufficiently large primes $p$, the cone over a Fermat-type Calabi--Yau hypersurface is perfectoid pure, revealing new and unexpected examples of perfectoid pure singularities.

\end{abstract}

\maketitle

\section{Introduction}

The theory of singularities in mixed characteristic has recently undergone remarkable development, as seen in works such as \cite{MSTWW}, \cite{BMPSTWW24}, and \cite{HLS}. 
This line of research has produced profound applications in both algebraic geometry and commutative algebra, including the mixed characteristic minimal model program \cite{BMPSTWW23}, \cite{TY23}, and Skoda-type theorems \cite{HLS}.

In this paper, we focus on the class of singularities known as \emph{perfectoid purity} \cite{p-pure}, 
which can be regarded as an analogue of $F$-purity in positive characteristic.  
It is known that if $(R,\m)$ is a Noetherian local ring with $p \in \m$, and if $R$ is a complete intersection such that $R/pR$ is quasi-$F$-split, 
then $R$ is perfectoid pure (\cite{Yoshikawa25}). 
On the other hand, in positive characteristic, 
the celebrated \emph{Fedder-type criterion} \cite{kty} provides a practical method for testing quasi-$F$-purity for complete intersections, 
leading to a wealth of explicit examples.

\medskip

In this paper, we establish a new numerical criterion that provides a sufficient condition for a hypersurface defined in an unramified regular local ring to be perfectoid pure.  
This criterion is expressed in terms of an explicitly computable sequence of integers, called the \emph{splitting-order sequence}.  
This invariant generalizes the quasi-$F$-split height introduced in \cite{Yoshikawa25}, and provides both a concrete criterion for perfectoid purity and a $p$-adic expansion of the perfectoid pure threshold.  
As a corollary, we show that for any regular local ring $R$, the perfectoid pure threshold $\ppt(R,p)$ is always a rational number.  
This rationality result may be viewed as a partial analogue of \cite{BMS}*{Theorem~3.1}, which proves that for a regular local ring of positive characteristic and any element $f$, the $F$-pure threshold \(\fpt(R,f)\) is rational.

We now fix notation and recall the notion of perfectoid purity.
Let $k$ be a perfect field and set
\[
A := W(k)[[x_1, \ldots, x_N]]
\]
with a Frobenius lift $\phi$ given by $\phi(x_i) = x_i^p$ for $1 \le i \le N$. 
For $a \in A$, set
\[
\Delta(a):=\frac{a^p-\phi(a)}{p}.
\]
Let $f \in A$ be such that $(p, f)$ forms a regular sequence.
We say that $A/f$ is \emph{perfectoid pure} if there exists a pure extension $A/f \to P$ to a perfectoid ring.
If $A/f$ is perfectoid pure, the \emph{perfectoid pure threshold  $\ppt(A/f,p)$ with respect to $p$} is defined as
\[
\sup \left\{\,\frac{i}{p^e} \geq 0 \ \middle|\
\begin{array}{l}
\text{there exist a perfectoid $A/f$-algebra $P$ and a $p^e$-th root $\varpi \in P$ of $p$}\\
\text{such that the map } A/f \to P \xrightarrow{\cdot \varpi^i} P \text{ is pure}
\end{array}
\right\},
\]
see \cref{good-perfd-cover} for details.  
It is known \cite{Yoshikawa25}*{Theorem~A} that if $\var{A/f}$ is quasi-$F$-split, then $A/f$ is perfectoid pure and $\ppt(A/f,p) > \frac{p-2}{p-1}$.

Next, we define the \emph{splitting-order sequence}.
Fix a generator \(u \in \Hom_{\var{A}}(F_*\var{A},\var{A})\) and set $\m^{[p^e]}:=(p, x^{p^e} \mid x \in \m)$.
For integers \(0 \le l_1, \ldots, l_{n-1} \le p-1\) and \(0 \le l_n \le p\),
we define ideals \(I(l_1,\ldots,l_n)\) of \(\var{A}\) inductively by
\begin{itemize}
  \item \(I(l_n) := f^{p-l_n}\var{A}\);
  \item once \(I(l_2,\ldots,l_n)\) is defined, set
  \[
    I(l_1,\ldots,l_n)
      := f^{p-l_1-1}\,u\bigl(F_*(\Delta(f)^{l_1}I(l_2,\ldots,l_n))\bigr)
         + f^{p-l_1}\var{A}.
  \]
\end{itemize}
We then define the \emph{splitting-order sequence} 
\(\bs(f) = (s_0,s_1,\ldots)\) of \(f\) with values in \(\{0,\ldots,p\}\)
inductively as follows:
\begin{itemize}
  \item set \(s_0 := 0\);
  \item once \(s_i = p\) for some \(i\), set \(s_n = p\) for all \(n > i\);
  \item suppose \(s_0,\ldots,s_{n-1}\) have been defined with \(s_1,\ldots,s_{n-1} \le p-1\),
  and define
  \[
  s_n := \max\bigl\{\,0 \le s \le p \,\bigm|\,
  I(s_1,\ldots,s_{n-1},s) \subseteq \m^{[p]}\,\bigr\}.
  \]
\end{itemize}
For alternative formulations of the splitting-order sequence, see \cref{splitting-order sequence},  
and for the equivalence among them, see \Cref{criterion}.
Furthermore, the splitting-order sequence does not depend on the choice of the Frobenius lift by \Cref{sos-vs-polygon}.

\noindent
The splitting-order sequence serves as a perfectoid analogue of a Fedder-type criterion in the theory of \(F\)-singularities.  
Our first main theorem shows that if all \(s_n \le p-1\), then \(A/f\) is perfectoid pure, and moreover it expresses $\ppt(A/f,p)$ explicitly in terms of \(\bs(f)\).

\begin{theoremA}[cf.~\cref{order-to-purity}]\label{intro:order-to-p-pure}
In the above setting, if \(s_n \le p-1\) for every \(n \ge 1\),
then \(A/f\) is perfectoid pure with 
\[
\ppt(A/f,p)=\sum_{n\ge1}\frac{p-1-s_n}{p^n}.
\]
\end{theoremA}

\noindent
A natural open problem is whether the converse direction also holds.

\begin{question}\label{intro:question}
Does the converse direction of Theorem~\ref{intro:order-to-p-pure} hold?
That is, if \(A/f\) is perfectoid pure, must we have \(s_n \le p-1\) for all \(n \ge 1\)?
\end{question}

\noindent
We obtain a partial answer to this question in the following theorem,
which completely resolves the case \(p=2\) and the case where \(A/f\) is regular.
Moreover, since the splitting-order sequence is invariant modulo \(p^2\)
(see \cref{rmk:modulo-p^2}), 
we deduce that perfectoid purity and the perfectoid pure threshold are 
also invariant modulo \(p^2\).

\begin{theoremA}[cf.~\cref{case:p-2/p-1}]\label{intro:p-pure-to-order}
In the above setting, we have:
\begin{enumerate}
  \item If \(A/f\) is perfectoid pure with 
  \(\ppt(A/f,p) \ge \frac{p-2}{p-1}\),
  then \(s_n \le 1\) for every \(n \ge 1\).
  \item If \(s_1 = \cdots = s_r = p-1\) and \(s_{r+1} = p\) for some \(r \ge 2\),
  then \(A/f\) is not perfectoid pure.
  \item If \(A/f\) is regular, then \(s_n \le p-1\) for every \(n \ge 0\).
\end{enumerate}
In particular, if \(p=2\) or \(A/f\) is regular, 
then perfectoid purity and \(\ppt(A/f,p)\) are invariant modulo \(p^2\); 
that is, for any \(f' \in A\) with \(f \equiv f' \pmod{p^2}\), 
we have that \(A/f\) is perfectoid pure if and only if so is \(A/f'\), 
and moreover \(\ppt(A/f,p) = \ppt(A/f',p)\).
\end{theoremA}

\noindent
Combining Theorems~\ref{intro:order-to-p-pure} and~\ref{intro:p-pure-to-order},
we obtain the following fundamental result, which shows that
$\ppt(R,p)$ is always a rational number for a regular local ring $R$, and that the set of perfectoid pure thresholds of $p$ on regular local rings satisfies the ascending chain condition.

\begin{theoremA}[\cref{rationality}]\label{intro:rationality}
Let $d$ be a non-negative integer.
Then the following two sets coincide:
\begin{itemize}
    \item the set $\mathcal{P}_{d,p}$ of $\ppt(R,p)$ such that $(R,\m)$ is a regular local ring of dimension $d$ with $p \in \m$ and $R/pR$ is $F$-finite, and
    \item the set of $\fpt(R,f)$ such that $(R,\m)$ is a regular $F$-finite local ring of dimension $d$ of characteristic $p$, and $0 \neq f \in \m$.
\end{itemize}
In particular, the set $\mathcal{P}_{d,p}$ is contained in $\Q$ and satisfies the ascending chain condition.
\end{theoremA}

Furthermore, when $R$ is a hypersurface of a regular local ring $A$ defined by $f$, 
the $F$-pure threshold $\fpt(A,f)$ can be interpreted as an invariant of $R$ that measures how close $R$ is to being $F$-pure. 
Indeed, it is known that $\fpt(A,f)=1$ if and only if $R$ is $F$-pure, and that $\fpt(A,f)$ is independent of the choice of the presentation $R \simeq A/(f)$.
The following theorem provides a more intrinsic description of this invariant in terms of the perfectoid pure thresholds of lifts of $R$.

\begin{theoremA}[\Cref{sup-ppt-fpt}]\label{intro:lift-sup-ppt}
Let $(A,\m)$ be a regular local ring of characteristic $p$, $f \in \m$, and $R:=A/f$.
Then we have
\[
\fpt(A,f)=\sup\{\ppt(\wt{R},p) \mid \text{$\wt{R}$ is a lift of $R$}\}.
\]
\end{theoremA}

\noindent
By Theorem~\ref{intro:lift-sup-ppt}, it is natural to define the following invariant of a ring $R$ in positive characteristic, which measures how close $R$ is to being $F$-pure:
\[
\sup\{\ppt(\wt{R},p) \mid \text{$\wt{R}$ is a lift of $R$}\}.
\]

\noindent
Finally, in Section~6 we compute explicit examples that illustrate these results.
In particular, we show that perfectoid purity appears in a surprisingly broad range of situations,  
including cones over Fermat-type Calabi--Yau hypersurfaces.

\begin{example}[\cref{Fermat-CY,Fermat-K3-surface,example3}]

\begin{enumerate}
    \item Let \(f = x_1^N + \cdots + x_p^N\). 
    If $p>N$, then  \(A/f\) is perfectoid pure.
    \item Let \(N=4\), $p \equiv 3 \pmod{4}$, and \(f=x_1^{4}+\cdots+x_{4}^{4}\). 
    Then \(A/f\) is perfectoid pure with \(\ppt(A/f,p)=2/(p^2-1)\).
    In particular, the lift of the Fermat-type K3 cone for \(p=3\) is perfectoid pure; 
    such a K3 surface has Artin invariant \(1\).
    \item Let \(p=2\), \(N=4\), and
    \[
    f = x_1^4 + x_2^4 + x_3^4 + x_4^4
      + x_1^2x_2^2 + x_1^2x_3^2 + x_2^2x_3^2
      + x_1x_2x_3(x_1 + x_2 + x_3)+px_1x_2x_3x_4,
    \]
    then \(A/f\) is perfectoid pure with \(\ppt(A/f,p)=0\).
    Moreover, 
    \[
    \Proj(\var{k[x_1,x_2,x_3,x_4]/f})
    \]
    is an RDP K3 surface of Artin invariant \(1\) by \cite{DK03}*{Theorem~1.1(vii)}.
\end{enumerate}
\end{example}

\noindent
These computations show that perfectoid purity can occur even in unexpected geometric settings, 
highlighting the richness of this class of singularities.  
Our approach therefore provides a new and effective computational framework for studying perfectoid analogues of $F$-singularity invariants.

\begin{ackn}
The author wishes to express his gratitude to Teppei Takamatsu and Linquan Ma  for valuable discussions.
He is also grateful to Vignesh Jagathese for helpful comments.
The author was supported by JSPS KAKENHI Grant number JP24K16889.
\end{ackn}

\section{Preliminaries}
In this section, we collect the basic definitions and notations used throughout the paper. We recall the conventions for $\delta$-rings, Frobenius lifts, and perfectoid structures, which will serve as the foundation for later sections.

\subsection{Notation and conventions
}\label{ss-notation}
In this subsection, we summarize the notation and terminology used throughout this paper.

\begin{enumerate}
\item We fix a prime number $p$ and set $\F_p := \Z/p\Z$.  
\item For a ring $R$, we set $\var{R}:=R/pR$.
\item For a ring $R$ of characteristic $p>0$, we denote by $F \colon R \to R$ the absolute Frobenius ring homomorphism, defined by $F(r) = r^p$ for all $r \in R$.
\item For a ring $R$ of characteristic $p>0$, we say that $R$ is {\em $F$-finite} if $F$ is finite. 
\item For a ring $R$ of characteristic $p>0$, we define the \emph{perfection} of $R$ as $R_{\mathrm{perf}} := \mathrm{colim}_F R$.
\item We use terminology \emph{$\delta$-ring} as defined in \cite{BS}*{Definition~2.1}. 
For a $\delta$-ring $A$, the delta map is denoted by $\delta$ and set $\Delta:=-\delta$.
\item We use the terminology \emph{perfectoid} as defined in \cite{BMS19}*{Definition~3.5}. For a semiperfectoid ring $R$, its \emph{perfectoidization} is denoted by $R_{\mathrm{perfd}}$ (cf.~\cite{BS}*{Section~7}).
\item Let $R$ be a ring and $f$ an element of $R$.
A sequence $\{f^{1/p^e}\}_{e \geq 0}$ is called \emph{a compatible system of $p$-power roots of $f$} if $(f^{1/p^{e+1}})^p=f^{1/p^e}$ for all $e \geq 0$ and $f^{1/p^0}=f$.
Furthermore, the ideal $(f^{1/p^e} \mid e \geq 0)$ of $R$ is denoted by $(f^{1/p^{\infty}})$.
\item For a ring $R$ and a ring homomorphism $\phi \colon R \to R$, we say that $\phi$ is a \emph{Frobenius lift} if $\phi(a) \equiv a^p \mod p$ for every $a \in R$.
\item For a ring homomorphism $f \colon A \to B$, we say that $f$ is \emph{pure} if for every $A$-module $M$, the homomorphism $M \to M \otimes_A B$ is injective. 
\end{enumerate}

\subsection{Frobenius lifts and perfectoid covers}
We construct certain perfectoid covers associated with a Frobenius lift on a regular local ring. This technical result will later be used to connect the splitting-order sequence to perfectoid purity.

\begin{proposition}\label{unramified-vs-F-lift}
Let \((A,\m)\) be a complete regular local ring with \(0 \ne p \in \m\).
Then \(A\) is unramified if and only if \(A\) admits a Frobenius lift. 
\end{proposition}

\begin{proof}
First, assume that \(A\) admits a Frobenius lift \(\phi\).
Since \(A\) is regular, \(\phi\) is flat.
Reducing modulo \(p\), we obtain that the Frobenius map
\[
F \colon \var{A} \longrightarrow F_*\var{A}
\]
is flat.
By Kunz’s theorem, this implies that \(\var{A}\) is regular.
Hence \(A\) is unramified.

Conversely, assume that \(A\) is unramified.
Then, by \cite{Matsumura}*{Theorem~29.7}, we have
\[
A \simeq C[[x_1, \ldots, x_N]]
\]
for some complete discrete valuation ring \(C\) with uniformizer \(p\).
By \cite{Matsumura}*{Theorem~29.2}, the ring \(C\) admits a Frobenius lift,
and therefore so does \(A\).
\end{proof}

\begin{proposition}\label{good-perfd-cover}
Let \((A,\m)\) be a regular ring with a finite Frobenius lift \(\phi\), and assume \(0 \neq p \in \m\).
Let \(f \in A\).
Then there exists a morphism of \(\delta\)-rings \(A \to B\) to a perfect prism \((B,(d))\) such that:
\begin{itemize}
  \item \(A \to B/(d)\) is faithfully flat;
  \item \(B/(d)\) contains compatible systems of \(p\)-power roots of \(f\) and \(p\);
  \item we have
  \[
    \ppt(A/f,p)
      = \sup\bigl\{
          \alpha \in \Z[1/p]_{\ge 0}
          \,\big|\,
          A/f \to B/(d,f^{1/p^\infty})
            \xrightarrow{\cdot p^\alpha}
            B/(d,f^{1/p^\infty})
          \text{ is pure}
        \bigr\}.
  \]
\end{itemize}
\end{proposition}

\begin{proof}
Since \(A\) is regular, its Frobenius lift \(\phi\) is faithfully flat.
Set \(A' := A[X_0]\) and define a ring endomorphism
\(\phi' \colon A' \to A'\) by \(\phi'|_A = \phi\) and \(\phi'(X_0) = X_0^p\);
then \(A \to A'\) is a morphism of \(\delta\)-rings.
Define
\[
  B' := (A'_{\perf})^{\wedge (p,d)}
     := \bigl(\varinjlim_{\phi'} A'\bigr)^{\wedge (p,d)},
\]
so that \((B',(d := X_0 - p))\) is a perfect prism.
By construction, \(A \to B'\) is a morphism of \(\delta\)-rings, and
\(A \to B'/(d)\) is \(p\)-completely faithfully flat.

By André’s flatness lemma, there exists a \((p,d)\)-faithfully flat morphism
of perfect prisms \((B',(d)) \to (B,(d))\)
such that \(B/(d)\) contains compatible systems of \(p\)-power roots of both \(f\) and \(p\).
Hence \(A \to B/(d)\) is also \(p\)-completely faithfully flat.
By \cite{Bhatt20}*{Lemma~5.15}, this implies that
\(A \to B/(d)\) is faithfully flat.

We now verify the third condition.
Let \(C\) be a perfectoid \(A/f\)-algebra.
By the proof of \cite{Yoshikawa25}*{Proposition~2.9}, it suffices to show that,
after replacing \(C\) by an \(\m\)-completely faithfully flat extension,
there exists an \(A\)-algebra homomorphism \(B/(d) \to C\).

After such a base change, we may assume that \(C\) contains a root of every monic polynomial and is \(\m\)-adically complete.
Choose a compatible system \(\{p^{1/p^e}\}_e\) of \(p\)-power roots of \(p\) in \(C\).
Replacing \(A\) by its \(\m\)-adic completion, we may assume that \(A\) is also \(\m\)-adically complete.
By further base change, we may assume that \(A/\m\) is perfect.

Define an \(A\)-algebra homomorphism \(A' \to C\) by sending \(X_0 \mapsto p\).
Since \(A\) is \(p\)-adically complete, by taking lifts of a \(p\)-basis we can choose elements \(x_1,\ldots,x_N \in A\) such that
\(\{\phi_*(x_1^{e_1}\cdots x_N^{e_N}) \mid 0 \le e_1,\ldots,e_N \le p-1\}\)
forms a basis of \(\phi_*A\) over \(A\).

Let \(A_1 \subseteq \phi_*A\) be the \(A\)-submodule generated by \(1, x_1, \ldots, x_1^{p-1}\), then it is $A$-subalgebra.
Then there exists a monic polynomial \(P \in A[T]\) of degree \(p\)
such that \(A[T]/(P) \simeq A_1\) as \(A\)-algebras.
Taking a root \(\alpha_1\) of \(P\) in \(C\), we obtain an \(A\)-algebra homomorphism
\[
A_1 \simeq A[T]/(P) \xrightarrow{T \mapsto \alpha_1} C.
\]
Repeating this process, we construct an \(A\)-algebra homomorphism
\(\phi_*A \to C\).

We now define an \(A'\)-algebra homomorphism by
\[
\phi'_*A' \simeq \phi_*A[X_0^{1/p}] \xrightarrow{X_0^{1/p} \mapsto p^{1/p}} C.
\]
Iterating this procedure yields an \(A'\)-algebra homomorphism \(B' \to C\)
sending \(d\) to \(0\).

Finally, the natural map
\(C \to C' := C \widehat{\otimes}_{B'/(d)} B/(d)\)
is \(p\)-completely faithfully flat, and there exists an induced \(A\)-algebra
homomorphism \(B \to C'\), as desired.
\end{proof}

\section{Splitting-order sequences}
In this section, we introduce the notion of a splitting-order sequence, a numerical invariant encoding how far an element is from being Frobenius split. We establish its basic properties and functoriality, and relate it to quasi-$F$-splitting.

\begin{notation}\label{notation:delta-ring}
Let \(A\) be a \(\delta\)-ring, and set \(\Delta := -\delta\).
For \(f \in A\) and integers \(0 \le l_1, \ldots, l_{n-1} \le p-1\) and \(0 \le l_n \le p\), 
we define an \(\var{A}\)-module \(\sQ{A,f}{l_1,\ldots,l_n}\) by
\[
\sQ{A,f}{l_1,\ldots,l_n}:=\frac{\var{A/f^{\,p}} \oplus \cdots \oplus F^{n-1}_*\var{A/f^p}}{\Bigl( (a_1 f^{\,l_1}, -a_1^{p} \Delta(f)^{\,l_1}+a_2f^{l_2},  \ldots, -a_{n-1}^p\Delta(f)^{l_{n-1}}+a_nf^{l_n})
      \ \Bigm|\ a_1,\ldots,a_n \in \var{A} \Bigr)}
\]


We also define \(\var{A}\)-module homomorphisms
\[
  \sPhi{A,f}{l_1,\ldots,l_n} \colon 
  \var{A/f^{\,l_1+1}} \longrightarrow \sQ{A,f}{l_1,\ldots,l_n},
  \qquad
  V \colon F^{\,n-1}_* \var{A/f^{\,l_n}} \longrightarrow \sQ{A,f}{l_1,\ldots,l_n},
\]
as follows.
\begin{itemize}
  \item Consider the \(\var{A}\)-linear map 
  \(\var{A} \to \sQ{A,f}{l_1,\ldots,l_n}\), \(a \mapsto (a,0,\ldots,0)\).
  In \(\sQ{A,f}{l_1,\ldots,l_n}\) we have
  \[
    (f^{\,l_1+1}, 0, \ldots, 0)
    = (0, f^{\,p} \Delta(f)^{\,l_1}, 0, \ldots, 0)
    = 0
  \]
  by the defining relations.
  Hence the map factors through \(\var{A/f^{\,l_1+1}}\), yielding
  \[
    \sPhi{A,f}{l_1,\ldots,l_n} \colon 
    \var{A/f^{\,l_1+1}} \longrightarrow \sQ{A,f}{l_1,\ldots,l_n}, 
    \qquad a \longmapsto (a,0,\ldots,0).
  \]
  \item Similarly, the \(\var{A}\)-linear map 
  \(F^{\,n-1}_*\var{A} \to \sQ{A,f}{l_1,\ldots,l_n}\), \(a \mapsto (0,\ldots,0,a)\),
  factors through \(F^{\,n-1}_* \var{A/f^{\,l_n}}\), giving
  \[
    V \colon 
    F^{\,n-1}_* \var{A/f^{\,l_n}} 
    \longrightarrow \sQ{A,f}{l_1,\ldots,l_n}, 
    \qquad a \longmapsto (0,\ldots,0,a).
  \]
\end{itemize}

When the base ring \(A\) is clear from the context, 
we simply write \(\sQ{f}{l_1,\ldots,l_n}\) and 
\(\sPhi{f}{l_1,\ldots,l_n}\) in place of 
\(\sQ{A,f}{l_1,\ldots,l_n}\) and \(\sPhi{A,f}{l_1,\ldots,l_n}\), respectively.
\end{notation}

\begin{remark}\label{rmk:module-structure}
With the notation as in \cref{notation:delta-ring}, fix $f \in A$.
Let $0 \le l_1, \ldots, l_{n-1} \le p-1$ and $0 \le l_n \le p$.
\begin{itemize}
    \item The $\var{A}$-module structure on $\sQ{f}{0,l_1,\ldots,l_n}$ 
    induces a natural $\var{A/f}$-module structure.
    Indeed, we have
    \[
      f(a_0,\ldots,a_n)
      = (0,\,f^p(a_0^p + a_1),\,f^{p^2}a_2,\,\ldots,\,f^{p^n}a_n)
      = 0
    \]
    in $\sQ{f}{0,l_1,\ldots,l_n}$.
    In particular, the map $\sPhi{f}{0,l_1,\ldots,l_n}$ 
    is an $\var{A/f}$-module homomorphism.
    
    \item If $\var{A}$ is $F$-finite, then $\sQ{f}{l_1,\ldots,l_n}$ 
    is a finite $\var{A}$-module.
\end{itemize}
\end{remark}

\begin{proposition}\label{mod-p^2}
With notation as in \cref{notation:delta-ring}, fix $f,f' \in A$, 
and integers $0 \le l_1, \ldots, l_{n-1} \le p-1$ and $0 \le l_n \le p$.
If $f \equiv f' \pmod{p^2}$, then 
\[
\sQ{f}{l_1,\ldots,l_n} \;\simeq\; \sQ{f'}{l_1,\ldots,l_n}.
\]
\end{proposition}

\begin{proof}
It suffices to show that $\Delta(f) \equiv \Delta(f') \pmod{p}$.
Write $f - f' = p^2 g$ for some $g \in A$.
Then
\[
\Delta(f) 
= \Delta(f' + p^2 g)
\equiv \Delta(f') + \Delta(p^2 g)
\equiv \Delta(f') \pmod{p},
\]
as desired.
\end{proof}

\begin{proposition}\label{compare-witt}
With notation as in \cref{notation:delta-ring}, fix \(f \in A\).
Then there exists an \(\var{A}\)-module isomorphism
\[
\var{W}_n(A/f) \;\xrightarrow{\ \simeq\ }\; \sQ{f}{1,\ldots,1},
\]
where the \(\var{A}\)-module structure on the left-hand side is given
via the ring homomorphism \(s_\phi \colon A \to W_n(A)\)
\textup{(cf.~\cite{Yoshikawa25}*{Proposition~3.2})}.
\end{proposition}

\begin{proof}
By \cite{Yoshikawa25}*{Theorem~3.7}, there is an \(A\)-module isomorphism
\[
\Psi \colon \var{W}_n(A)
  \xrightarrow{\ \simeq\ }
  \var{A} \oplus F_*\var{A} \oplus \cdots \oplus F^{n-1}_*\var{A},
\]
defined by
\[
\Psi(a_0,\ldots,a_{n-1})
  = \bigl(a_0,\,
     \Delta_1(a_0)+a_1,\,
     \ldots,\,
     \Delta_{n-1}(a_0)+\cdots+a_{n-1}\bigr).
\]

We define a map
\[
\Psi' \colon \var{W}_n(A) \longrightarrow \sQ{f}{1,\ldots,1}
\]
by
\[
\Psi'(a_0,\ldots,a_{n-1})
  = \bigl(a_0,\,
     -(\Delta_1(a_0)+a_1),\,
     \ldots,\,
     (-1)^{n-1}(\Delta_{n-1}(a_0)+\cdots+a_{n-1})\bigr).
\]
We show that \(\Psi'\) induces an isomorphism after passing to the quotient by \(W_n(f)\).

For any \((a_0,\ldots,a_{n-1}) \in \var{W}_n(A)\), we compute:
\begin{align*}
\Psi'(a_0f,a_1f,\ldots,a_{n-1}f)
  &= \bigl(a_0f,\,
     -a_0^p\Delta_1(f)-a_1f,\,
     \ldots,\,
     (-1)^{n-1}(a_0^{p^{n-1}}\Delta_{n-1}(f)+\cdots+a_{n-1}f)\bigr) \\
  &\overset{(\star_1)}{=}
     \bigl(a_0f,\,
       -a_0^p\Delta_1(f)-a_1f,\,
       \ldots,\,
       (-1)^{n-1}(a_{n-1}^p\Delta_1(f)+a_nf)\bigr) \\
  &\overset{(\star_2)}{=}
     (0,\ldots,0),
\end{align*}
where \((\star_1)\) follows from \cite{Yoshikawa25}*{Theorem~3.6} and
\((\star_2)\) from the definition of \(\sQ{f}{1,\ldots,1}\).

Hence \(\Psi'\) induces an \(\var{A}\)-module homomorphism
\[
\Psi'' \colon \var{W}_n(A/f) \longrightarrow \sQ{f}{1,\ldots,1}.
\]
By construction, \(\Psi'\) is surjective, so \(\Psi''\) is surjective as well.

To prove injectivity, let
\(\alpha = (a_0,\ldots,a_{n-1}) \in \var{W}_n(A)\)
satisfy \(\Psi'(\alpha) = 0\).
Then \(a_0 \in (f)\).
Replacing \(a_0\) by zero, we proceed inductively to see that
each \(a_i \in (f)\) for all \(i\).
Therefore \(\alpha \in \var{W}_n(fA)\),
and hence \(\Psi''\) is injective.

Thus \(\Psi''\) is an isomorphism of \(\var{A}\)-modules, as claimed.
\end{proof}



\begin{proposition}\label{functoriality}
With notation as in \cref{notation:delta-ring}, let 
$\varphi \colon A \to B$ be a morphism of $\delta$-rings.  
Take $f \in A$ and $g \in B$ such that $\varphi(f) \in (g)$.
Then, for all integers $0 \le l_1,\ldots,l_{n-1} \le p-1$ and $0 \le l_n \le p$, 
the map $\varphi$ induces an $\var{A}$-module homomorphism
\[
\sQ{f}{l_1,\ldots,l_n} \;\longrightarrow\; \sQ{g}{l_1,\ldots,l_n}.
\]
\end{proposition}

\begin{proof}
It suffices to show that, for $0 \le l \le p-1$, one has
\[
(\varphi(f^{\,l}), -\varphi(\Delta(f)^{\,l})) 
  \in 
  \bigl( (b g^{\,l}, -b^{p}\Delta(g)^{\,l}) \ \bigm|\ b \in \var{B} \bigr)
  \subseteq \var{B/g^p} \oplus F_*\var{B/g^p}.
\]
Write $\varphi(f) = g h$.  
Then, since $\varphi$ is a morphism of $\delta$-rings,
\[
(\varphi(f^{\,l}), -\varphi(\Delta(f)^{\,l})) 
= (h^{l} g^{l}, -\Delta(gh)^{l})
= (h^{l} g^{l}, -h^{pl}\Delta(g)^{l}),
\]
where the last equality holds because 
\((0,-g^{p}\Delta(h)) = 0\) in 
\(\var{B/g^p} \oplus F_*\var{B/g^p}\).
\end{proof}

\begin{proposition}\label{exact-seq}
With notation as in \cref{notation:delta-ring}, fix \(f \in A\).
Assume that \(p, f\) form a regular sequence.
For \(n \ge 2\), \(0 \le l_1, \ldots, l_{n-1} \le p-1\), and \(0 \le l_n \le p\), 
there is an exact sequence
\[
0 \;\longrightarrow\;
F^{n-1}_*\var{A/f^{l_n}}
\xrightarrow{\,V\,}
\sQ{f}{l_1,\ldots,l_n}
\xrightarrow{\,R\,}
\sQ{f}{l_1,\ldots,l_{n-1}}
\;\longrightarrow\; 0,
\]
where the map \(R\) is given by 
\[
R(a_1,\ldots,a_n) = (a_1,\ldots,a_{n-1}).
\]
\end{proposition}

\begin{proof}
It suffices to show that \(V\) is injective.
Take \(a \in \var{A}\) such that \((0,\ldots,0,a) = 0\) in \(\sQ{f}{l_1,\ldots,l_n}\).
Then there exist \(a_1,\ldots,a_n \in \var{A}\) satisfying
\[
(0,\ldots,0,a)
=
(a_1 f^{l_1},\, -a_1^p \Delta(f)^{l_1} + a_2 f^{l_2},\, \ldots,\, 
-a_{n-1}^p \Delta(f)^{l_{n-1}} + a_n f^{l_n})
\]
in \(\var{A/f^p} \oplus \cdots \oplus F^{n-1}_*\var{A/f^p}\).

Since \(p,f\) form a regular sequence, we have \(a_1 \in f^{p - l_1} \var{A}\).
Because \(l_1 \le p - 1\), it follows that \(a_1^p \in f^p \var{A}\); in particular, 
\(a_2 f^{l_2} = 0\) in \(\var{A/f^p}\) when \(n \ge 3\).
Repeating this argument inductively, we obtain 
\(a = a_n f^{l_n}\) in \(\var{A/f^p}\), hence \(a \in f^{l_n} \var{A}\), 
as desired.
\end{proof}

\begin{definition}\label{splitting-order sequence}
With notation as in \cref{notation:delta-ring}, fix $f \in A$. 
We define a \emph{splitting-order sequence $\bs(f)=(s_i)_{i\ge 0}$  of $f$} with values in $\{0,\ldots,p\}$ inductively as follows:
\begin{itemize}
    \item Set $s_0 := 0$.
    \item Once $s_i = p$ for some $i$, set $s_n = p$ for all $n > i$.
    \item Suppose $s_0,\ldots,s_{n-1}$ have been defined for some $n \ge 1$ with $s_0,\ldots,s_{n-1} \le p-1$.
    We define
    \[
      s_n := \max\bigl\{\,0 \le s \le p \ \bigm|\ 
      \text{$\sPhi{f}{s_0,\ldots,s_{n-1},s}$ is not pure as a $\var{A}$-module map}\,\bigr\}.
    \]
\end{itemize}

\end{definition}

\begin{remark}\label{rmk:modulo-p^2}
\begin{itemize}
    \item In \Cref{sos-vs-polygon}, we prove that the splitting-order sequence does not depend on the choice of Frobenius lift.
    \item If $\sPhi{f}{s_0,\ldots,s_{n-1}}$ is not pure, then 
    $\sPhi{f}{s_0,\ldots,s_{n-1},0}$ is also not pure, since
    \[
      \sQ{f}{s_0,\ldots,s_{n-1},0} \simeq \sQ{f}{s_0,\ldots,s_{n-1}}.
    \]
    In particular, the set appearing in \cref{splitting-order sequence} is non-empty.
    
    \item If $(A,\m)$ is Noetherian local, then 
    $\sPhi{f}{0,l_1,\ldots,l_n}$ is pure if and only if the induced map
    \[
      \sPhi{f}{0,l_1,\ldots,l_n} \colon 
      H^d_\m(\var{A/f}) \longrightarrow H^d_\m(\sQ{f}{0,l_1,\ldots,l_n})
    \]
    is injective, where $d := \dim(\var{A/f})$.  
    This equivalence holds because $\sPhi{f}{0,l_1,\ldots,l_n}$ is 
    an $\var{A/f}$-module homomorphism by \cref{rmk:module-structure}.
    
    \item In the setting of \cref{splitting-order sequence}, 
    the splitting-order sequence is invariant modulo $p^2$.  
    Indeed, if $f \equiv f' \pmod{p^2}$, then $\bs(f) = \bs(f')$ by \cref{mod-p^2}.
\end{itemize}
\end{remark}

\begin{proposition}\label{qfs-splitting-order}
Let \((A,\m)\) be a Noetherian local ring with a finite Frobenius lift \(\phi\), and assume that \(p \in \m\) and $A$ is $p$-torsion-free.
Let \(f \in A\) and let \(\bs(f)\) denote the splitting-order sequence of \(f\).
Then \(A/f\) is quasi-\(F\)-split of height \(h\) 
if and only if \(s_1 = \cdots = s_{h-1} = 1\) and \(s_h = 0\).
\end{proposition}

\begin{proof}
We first observe that there is an isomorphism
\[
\sQ{f}{s_0,\ldots,s_n}
  \xrightarrow{\ \simeq\ }
  F_*\sQ{f}{s_1,\ldots,s_n},
  \qquad
  (a_0,\ldots,a_n) \longmapsto (a_0^p + a_1,\ldots,a_n).
\]
Composing \(\sPhi{f}{s_0,\ldots,s_n}\) with the above map yields
\[
\var{A/f}
  \longrightarrow
  F_*\sQ{f}{s_1,\ldots,s_n},
  \qquad
  a \longmapsto (a^p,0,\ldots,0),
\]
which agrees with the morphism \(\Phi_{A/f,n}\)
after composing with the isomorphism in \cref{compare-witt}.
Therefore, the claim follows directly from the definitions of the sequence \(\bs(f)\) and quasi-$F$-splitting.
\end{proof}

\begin{proposition}\label{completion}
Let \((A,\m)\) be a Noetherian local ring with a finite Frobenius lift \(\phi\),
and assume that \(p \in \m\).
Let \(\iota \colon A \to \widehat{A}\) denote the \(\m\)-adic completion.
Then \(\phi\) induces a Frobenius lift \(\widehat{\phi}\) on \(\widehat{A}\).
We equip \(A\) and \(\widehat{A}\) with the \(\delta\)-ring structures
corresponding to \(\phi\) and \(\widehat{\phi}\), respectively.
Take \(f \in A\) such that \(p,f\) form a regular sequence.
Then  \(\bs(f) = \bs(\iota(f))\).
\end{proposition}

\begin{proof}
Since \(\var{A}\) is \(F\)-finite, we have the following commutative diagram:
\[
\begin{tikzcd}[column sep=2.5cm]
  \bigl(\var{A/f} \otimes_{\var{A}} \var{\widehat{A}}\bigr)
    \arrow[r, "{\sPhi{A,f}{l_1,\ldots,l_n} \otimes \var{\widehat{A}}}"]
    \arrow[d, "\simeq"']
  & \bigl(\sQ{A,f}{l_1,\ldots,l_n} \otimes_{\var{A}} \var{\widehat{A}}\bigr)
    \arrow[d, "\simeq"] \\
  \var{\widehat{A}/\iota(f)}
    \arrow[r, "{\sPhi{\widehat{A},\iota(f)}{l_1,\ldots,l_n}}"]
  & \sQ{\widehat{A},\iota(f)}{l_1,\ldots,l_n}
\end{tikzcd}
\]
for \(0 \le l_1,\ldots,l_{n-1} \le p-1\) and \(0 \le l_n \le p\).
The desired assertion follows immediately from the diagram.
\end{proof}

\section{Perfectoid purity versus  splitting-order sequences}
This section establishes the relationship between perfectoid purity and the boundedness of the splitting-order sequence. The key result shows that the perfectoid $p$-purity threshold can be expressed in terms of the splitting-order sequence.

\begin{lemma}\label{CM-inj}
With notation as in \cref{notation:delta-ring}, fix \(f \in A\),
\(n \ge 1\), \(0 \le l_1,\ldots,l_{n-1} \le p-1\), and \(0 \le l_n \le p\).
Assume that \((A,\m)\) is a Noetherian local ring such that \(p,f\) form a regular sequence.
Suppose further that \(A\) is Cohen--Macaulay.
Then \(H^i_\m(\sQ{f}{l_1,\ldots,l_n}) = 0\) for all \(i < d := \dim(\var{A/f})\),
and the morphism
\[
H^d_\m(F^{n-1}_*\var{A/f^{l_n}}) \longrightarrow H^d_\m(\sQ{f}{l_1,\ldots,l_n})
\]
induced by \(V\) is injective.
\end{lemma}

\begin{proof}
Since \(A\) is Cohen--Macaulay and \(p,f\) form a regular sequence, 
each \(\var{A/f^l}\) is Cohen--Macaulay for \(l \ge 1\).
The desired vanishing and injectivity then follow from the exact sequence in \cref{exact-seq}.
\end{proof}

\begin{proposition}\label{perfectoid-Q}
Let $(A,(d))$ be a perfect prism, and set $T:=\var{d} \in \var{A}$. Then we obtain the isomorphism
\[
\var{A}/(T^{l_0+l_1/p+\cdots+l_n/p^n}) \simeq \sQ{d}{l_0,l_1,\ldots,l_n}\ \, a \mapsto (a,0,\cdots,0)
\]
for $0 \leq l_0,\ldots,l_{n-1} \leq p-1$ and $0 \leq l_n \leq p$.
\end{proposition}

\begin{proof}
We prove the assertion by induction on $n$.
If $n=0$, it is clear.
For $n \geq 1$, we have
\[
\sQ{d}{l_0,\ldots,l_n} \simeq \var{A}/T^{l_0} \oplus F_*\var{A}/(T^{l_1+\cdots+l_n/p^{n-1}})/\{(aT^{l_0},-a^p\Delta(d)^{l_0}) \mid a \in \var{A}\}=:P.
\]
First, the homomorphism
\[
\var{A}/(T^{l_0+l_1/p+\cdots+l_n/p^n}) \to P
\]
is well-defined, indeed, we have
\[
(T^{l_0+l_1/p+\cdots+l_n/p^n},0) =(0,T^{l_1+\cdots+l_n/p^{n-1}}\Delta(d)^{l_0})=0.
\]
Furthermore, the inverse homomorphism
\[
P \to \var{A}/(T^{l_0+l_1/p+\cdots+l_n/p^n})
\]
is defined by
\[
(a,b) \mapsto a-F^{-1}(b)\Delta(d)^{-l_0}T^{l_0}.
\]
Then it is well-defined, as desired.
\end{proof}

\begin{lemma}\label{f-versus-y}
With notation as in \cref{notation:delta-ring}, fix \(f \in A\), \(n \ge 1\), 
\(0 \le l_1,\ldots,l_{n-1} \le p-1\), and set \(l_n := p\). 
Let \((B,(d))\) be a perfect prism with Frobenius lift \(\phi\), 
and let \(\varphi \colon A \to B\) be a morphism of \(\delta\)-rings.
Assume there exists \(y \in B\) such that \(\phi(y)=y^p\) and \(y-f \in d\var{B}\).
Then for every element
\[
\alpha \in 
\Ker\!\Bigl(
\var{A} \oplus F_*\var{A} \oplus \cdots \oplus F^n_*\var{A}
\longrightarrow \sQ{A,f}{0,l_1,\ldots,l_n}
\Bigr),
\]
we have
\[
\alpha\,y^{1-1/p^{n-1}} \in f\,\sQ{B,d}{0,l_1,\ldots,l_{n-1},1}.
\]
\end{lemma}

\begin{proof}
There exist elements \(a_0,\ldots,a_n \in \var{A}\) such that
\[
\alpha = 
(a_0,\,
 -a_0^p + a_1 f^{l_1},\,
 -a_1^p \Delta(f)^{l_1} + a_2 f^{l_2},\,
 \ldots,\,
 -a_{n-1}^p \Delta(f)^{l_{n-1}} + a_n f^p).
\]
We first note that
\[
(0,\ldots,0,a_n f^p)\,y^{1-1/p^{n-1}}
  = (0,\ldots,0,a_n f^p y^{p^n-p})
  = (0,\ldots,0,a_n f^{p^n})
\]
in \(\sQ{B,d}{0,l_1,\ldots,l_{n-1},1}\).
Hence we may assume \(a_n = 0\).
If \(n=1\), then
\[
\alpha=(a_0,-a_0^p)=0 \text{ in } \sQ{B,f}{0,1},
\]
as desired.
Thus we may further assume \(n \ge 2\).

Next, we reduce to the case \(a_{n-1}=0\).
Consider the \(\var{B}\)-module homomorphism
\[
F^{n-1}_*\sQ{B,d}{l_{n-1},1} \to \sQ{B,d}{0,l_1,\ldots,l_{n-1},1},\qquad 
(a,b) \mapsto (0,\ldots,0,a,b).
\]
Let 
\[
\beta := (a_{n-1}f^{l_{n-1}},-a_{n-1}^p\Delta(f)^{l_{n-1}})y^{p^{n-1}-1} 
  \in \sQ{B,d}{l_{n-1},1}.
\]
Write \(f \equiv y + u d \pmod{pB}\) for some \(u \in B\).
Then
\[
\Delta(f) \equiv u^p \Delta(d) \pmod{d\var{B}}.
\]
Thus
\[
\beta =
(a_{n-1} f^{l_{n-1}}, - (a_{n-1} u^{l_{n-1}})^p \Delta(d)^{l_{n-1}})y^{p^{n-1}-1}
  =(a_{n-1}y^{p^{n-1}-1}(f^{l_{n-1}}-(ud)^{l_{n-1}}),0)
\]
in $\sQ{B,d}{l_{n-1},1}$.
Since \(f^{l_{n-1}} - (u d)^{l_{n-1}} \in y\var{B}\),
we have
\(a_{n-1}y^{p^{n-1}-1}(f^{l_{n-1}}-(ud)^{l_{n-1}})
  \in y^{p^{n-1}}\var{B}\).
Because \(n-1 \ge 1\), this implies
\(\beta \in f^{p^{n-1}}\sQ{B,d}{l_{n-1},1}\), and hence
\[
(0,\ldots,0,a_{n-1}f^{l_{n-1}},-a_{n-1}^p\Delta(f)^{l_{n-1}}) 
  \in f\,\sQ{B,d}{0,l_1,\ldots,l_{n-1},1}.
\]
Thus we may assume \(a_{n-1}=0\).

Next, set \(l_0=0\).
For each \(0 \le i \le n-2\), define a 
\(\var{B}\)-linear homomorphism
\[
\psi_i \colon
F^i_* \sQ{B,d}{l_i,p}
  \longrightarrow
  \sQ{B,d}{0,l_1,\ldots,l_n}
\]
as follows:
for \((a,b) \in F^i_* \sQ{B,d}{l_i,p}\), let 
\(\psi_i(a,b) = (c_0,\ldots,c_n)\),
where
\[
c_j =
\begin{cases}
a & \text{if } j=i,\\
b & \text{if } j=i+1,\\
0 & \text{otherwise.}
\end{cases}
\]

It suffices to show that
\[
\beta_i := 
(a_i f^{l_i}, -a_i^p \Delta(f)^{l_i})
\,y^{p^i - 1/p^{n-i}} 
\in f^{p^i}\sQ{B,d}{l_i,p}
\quad\text{for all } 0 \le i \le n-1.
\]
If \(i=0\), then \(\beta_0 = 0\) in \(\sQ{B,d}{l_0,p}\) since \(l_0=0\).

Now assume \(n-2 \ge i \ge 1\).
Then
\[
\Delta(f) \equiv u^p \Delta(d) \pmod{(y,d^p)\var{B}}.
\]
It follows that
\[
a_i^p y^{p^{i+1} - 1/p^{n-i-2}} \Delta(f)^{l_i}
\equiv
(a_i y^{p^i - 1/p^{n-i-1}} u^{l_i})^p \Delta(d)^{l_i}
\pmod{(y^{p^{i+1}}, d^p)\var{B}},
\]
since \(n-i-2 \ge 0\).
Moreover, as \(i+1 \ge 1\), we have
\(y^{p^{i+1}} \equiv f^{p^{i+1}} \pmod{d^p \var{B}}\),
and therefore
\[
\beta_i 
\equiv
(a_i f^{l_i}, - (a_i u^{l_i})^p \Delta(d)^{l_i})
\,y^{p^i - 1/p^{n-i-1}}
\pmod{f^{p^i} \sQ{B,d}{l_i,p}}.
\]
Hence
\[
\beta_i 
\equiv 
(a_i f^{l_i} - a_i (u d)^{l_i}, 0)
\,y^{p^i - 1/p^{n-i-1}}
\pmod{f^{p^i} \sQ{B,d}{l_i,p}}.
\]

Since \(f^{l_i} - (u d)^{l_i} \in y\var{B}\),
we have
\(a_i y^{p^i - 1/p^{n-i-1}}(f^{l_i} - (u d)^{l_i})
  \in y^{p^i}\var{B}\).
Because \(i \ge 1\), this implies
\(\beta_i \in f^{p^i}\sQ{B,d}{l_i,p}\), as desired.
\end{proof}

\begin{lemma}\label{f-versus-y-var}
With notation as in \cref{notation:delta-ring}, fix $f \in A$, $n \ge 1$,
and integers $0 \le l_1,\ldots,l_{n-1} \le p-1$, $0 \le l_n \le p$.
Let $(B,(d))$ be a perfect prism with Frobenius lift $\phi$,
and let $\varphi \colon A \to B$ be a morphism of $\delta$-rings.
Assume that there exists $y \in B$ such that $\phi(y)=y^p$ and $y-f \in d\var{B}$.
Then, for every element
\[
\alpha \in 
\Ker\!\Bigl(
\var{A} \oplus F_*\var{A} \oplus \cdots \oplus F^n_*\var{A}
\longrightarrow \sQ{A,f}{0,l_1,\ldots,l_n}
\Bigr),
\]
we have
\[
\alpha\,y^{1-1/p^{n}} \in f\,\sQ{B,d}{0,l_1,\ldots,l_{n-1},l_n}.
\]
\end{lemma}

\begin{proof}
Choose elements $a_0,\ldots,a_n \in \var{A}$ such that
\[
\alpha = 
(a_0,\,
 -a_0^p + a_1 f^{l_1},\,
 -a_1^p \Delta(f)^{l_1} + a_2 f^{l_2},\,
 \ldots,\,
 -a_{n-1}^p \Delta(f)^{l_{n-1}} + a_n f^{l_n}).
\]
By the argument in the proof of \Cref{f-versus-y}, we may assume that
$a_0=\cdots=a_{n-1}=0$.
Then
\[
(0,\ldots,0,a_n f^{l_n})\,y^{1-1/p^{n}}
= (0,\ldots,0,a_n f^{l_n} y^{p^n-1}).
\]
Since $f^{l_n} \in (y,d^{l_n})\var{B}$, it follows that
\[
(0,\ldots,0,a_n f^{l_n} y^{p^n-1})
\in f\sQ{B,d}{0,l_1,\ldots,l_{n-1},1},
\]
and hence
\[
(0,\ldots,0,a_n f^{l_n} y^{p^n-1})
\in f\sQ{B,d}{0,l_1,\ldots,l_{n-1},l_n},
\]
as desired.
\end{proof}

\begin{lemma}\label{good-perfd}
Let \((A,\m)\) be a regular local ring equipped with a finite Frobenius lift \(\phi\),
and assume that \(0 \neq p \in \m\).
We endow \(A\) with the \(\delta\)-ring structure induced by \(\phi\).
Let \(f \in A\).
Then there exist a perfect prism \((B,(d))\), a morphism of \(\delta\)-rings
\(\varphi \colon A \to B\), and an element \(y \in B\) such that
\[
\phi(y)=y^p, \qquad y-f \in (d), \qquad\text{and}\qquad A \to B/(d) \text{ is faithfully flat.}
\]
In particular, \(B/(d)\) is a Cohen--Macaulay \(A\)-algebra and
\(A \to B/(d)\) is pure.
\end{lemma}

\begin{proof}
Take a perfect prism \((B,(d))\) as constructed in \cref{good-perfd-cover};
then \(A \to B/(d)\) is faithfully flat.
Define \(y \in B\) to be the element corresponding to \([f^\flat] \in W((B/(d))^\flat)\)
under the canonical isomorphism \(W((B/(d))^\flat) \simeq B\).
By construction we have \(\phi(y) = y^p\) and \(y - f \in (d)\), as desired.
\end{proof}

\begin{theorem}\label{sos-vs-polygon}
Let $(A,\m)$ be a regular local ring with $p \in \m$ and admitting a finite Frobenius lift $\phi$.
Let $f \in A$ be such that $f,p$ form a regular sequence.
Let $\bs(f)=(s_i)_{i \geq 0}$ denote the splitting-order sequence of $f$.
For each $n \geq 1$, set
\[
t_n:=\frac{p-1-s_1}{p}+\cdots+\frac{p-1-s_n}{p^n}.
\]
Then, for every $n \geq 1$ with $s_{n-1} \leq p-1$, we have
\[
s_{n}
=
\max\Bigl\{1 \leq s \leq p \ \Bigm|\ (A,f^{1-\frac{s}{p^n}} p^{t_{n-1}+\frac{p-1}{p^n}})\ \text{is not perfectoid pure}\Bigr\}.
\]
In particular, the splitting-order sequence does not depend on the choice of the Frobenius lifts on $A$.
\end{theorem}

\begin{proof}
Set $d:=\dim(\var{A/f})$, and choose a nonzero element $\eta \in H^d_\m(\var{A/f})$ in the socle.
Fix $s \leq s_n$. 
By definition of $\bs(f)$, the map $\sPhi{f}{s_0,\ldots,s_{n-1},s}$ is not pure.

Put $\tau:=\eta/f\in H^{d+1}_\m(\var{A})$.
Choose a regular sequence $x_1,\ldots,x_d\in A$ such that $x_1,\ldots,x_d,f$ form a system of parameters, and choose $a\in\var{A}$ with
\[
\tau=\Bigl[\frac{a}{x_1\cdots x_d f}\Bigr].
\]
Then there exists $c\in\var{A}$ such that
\[
(a,0,\ldots,0)\equiv(0,\ldots,0,f^{s}c)
\pmod{(x_1,\ldots,x_d)\,\sQ{f}{s_0,\ldots,s_{n-1},p}}.
\]
Hence there exists
\[
\alpha\in
\Ker\!\Bigl(
\var{A}\oplus F_*\var{A}\oplus\cdots\oplus F^n_*\var{A}
\to \sQ{A,f}{0,s_1,\ldots,s_n+1}
\Bigr)
\]
such that
\[
(a,0,\ldots,0)+\alpha\equiv(0,\ldots,0,f^{s}c)
\pmod{(x_1,\ldots,x_d)(\var{A}\oplus\cdots\oplus F^n_*\var{A})}.
\]

Let $(B,(d))$ be a perfect prism, let $\varphi\colon A\to B$ be a morphism of $\delta$-rings, and let $y\in B$ be as in \Cref{good-perfd}.
By \Cref{f-versus-y} (note that $s\le p$), we obtain
\[
(a,0,\ldots,0)\,y^{1-s/p^n}
\equiv
(0,\ldots,0,f^{s}c)\,y^{1-s/p^n}
\pmod{(x_1,\ldots,x_d,f)\sQ{B,d}{s_0,\ldots,s_{n-1},1}}.
\]
Let $\tau'$ denote the image of $\tau$ under the natural map
\[
H^{d+1}_\m(\var{A})\longrightarrow H^{d+1}_\m(\var{B/(d)}).
\]
Then
\begin{align*}
\sPhi{B,d}{s_0,\ldots,s_{n-1},1}(\tau' y^{1-s/p^n}) 
&=\sPhi{B,d}{s_0,\ldots,s_{n-1},1}
  \Bigl(\Bigl[\frac{a}{x_1\cdots x_d f}\Bigr]y^{1-s/p^n}\Bigr) \\
&=\Bigl[\frac{(a,0,\ldots,0)}{x_1\cdots x_d f}\Bigr]y^{1-s/p^n} \\
&=\Bigl[\frac{(0,\ldots,0,f^{s}c)}{x_1\cdots x_d f}\Bigr]y^{1-s/p^n} \\
&=V\Bigl(\Bigl[\frac{f^{s}c}{(x_1\cdots x_d)^{p^n}f^{p^n}}\Bigr]f^{p^n-s}\Bigr)
=0.
\end{align*}

Therefore,
\[
f^{1-\frac{s}{p^n}}p^{t_{n-1}+\frac{p-1}{p^n}}\tau' =0,
\]
because
\[
B/(d,p^{t_{n-1}+\frac{p-1}{p^n}})
\simeq Q_{B,d,(s_0,s_1,\ldots,s_{n-1},1)}
\]
by \Cref{perfectoid-Q}.
It follows from the argument in the proof of \cref{good-perfd-cover} that the pair
\[
(A,f^{1-\frac{s}{p^n}} p^{t_{n-1}+\frac{p-1}{p^n}})
\]
is not perfectoid pure.

Next, assume $s_n \leq p-1$.
It suffices to show that
\[
\sPhi{B,d}{s_0,\ldots,s_{n-1},1}
(\tau' y^{1-(s_n+1)/p^n}) \neq 0.
\]
By \Cref{exact-seq}, there exists $\eta'_n \in H^d_\m(\var{A/f^{\,s_n+1}})$ such that
\[
V(\eta'_n)=\sPhi{f}{s_0,\ldots,s_{n-1},s_n+1}(\eta).
\]
Since $\sPhi{f}{s_0,\ldots,s_{n-1},s_n}(\eta)=0$, 
\Cref{CM-inj} implies that
\[
\bigl(H^d_\m(\var{A/f^{\,s_n+1}})
\to H^d_\m(\var{A/f^{\,s_n}})\bigr)(\eta'_n)=0.
\]
Hence there exists $\eta_n\in H^d_\m(\var{A/f})$ such that
$f^{s_n}\eta_n=\eta'_n$.
Put $\tau_n:=\eta_n/f\in H^{d+1}_\m(\var{A})$.
Replacing $x_1,\ldots,x_d$ by a suitable regular sequence, there exist $a,b\in\var{A}$ such that
\[
\tau=\Bigl[\frac{a}{x_1\cdots x_d f}\Bigr],\qquad
\tau_n=\Bigl[\frac{b}{(x_1\cdots x_d)^{p^n} f}\Bigr].
\]
Then there exists $c\in\var{A}$ such that
\[
(a,0,\ldots,0)\equiv
(0,\ldots,0,f^{s_n}b+f^{s_n+1}c)
\pmod{(x_1,\ldots,x_d)\,\sQ{f}{s_0,\ldots,s_{n-1},p}}.
\]
Arguing as above, we obtain
\[
(a,0,\ldots,0)\,y^{1-(s_n+1)/p^n}
\equiv
(0,\ldots,0,f^{s_n}b+f^{s_n+1}c)\,y^{1-(s_n+1)/p^n}
\]
modulo $(x_1,\ldots,x_d,f)\,\sQ{B,d}{s_0,\ldots,s_{n-1},1}$.
Therefore, we have
\begin{align*}
V(\tau'_n)
&=V\!\Bigl(\Bigl[\frac{b}{(x_1\cdots x_d)^{p^n}f}\Bigr]\Bigr)
 =V\!\Bigl(\Bigl[\frac{f^{s_n}b}{(x_1\cdots x_d f)^{p^n}}\Bigr]y^{p^n-s_n-1}\Bigr) \\
&=\Bigl[\frac{(0,\ldots,0,f^{s_n}b)}{x_1\cdots x_d f}\Bigr]\,y^{1-(s_n+1)/p^n}
 =\Bigl[\frac{(a,0,\ldots,0)}{x_1\cdots x_d f}\Bigr]\,y^{1-(s_n+1)/p^n} \\
&=\sPhi{B,d}{s_0,\ldots,s_{n-1},1}\!\Bigl(\Bigl[\frac{a}{x_1\cdots x_d f}\Bigr]y^{1-(s_n+1)/p^n}\Bigr)
 =\sPhi{B,d}{s_0,\ldots,s_{n-1},1}(\tau' y^{1-(s_n+1)/p^n}),
\end{align*}
where $\tau'_n$ denotes the image of $\tau_n$ in
$H^{d+1}_\m(\var{B/(d)})$.

Since $B/(d)$ is Cohen--Macaulay, the map
\[
V\colon H^{d+1}_\m(\var{B/(d)})
\longrightarrow
H^{d+1}_\m(\sQ{B,d}{s_0,\ldots,s_{n-1},1})
\]
is injective.
Moreover, because $A\to B/(d)$ is pure, we have $V(\tau'_n)\neq 0$.
Therefore,
\[
\sPhi{B,d}{s_0,\ldots,s_{n-1},1}
(\tau' y^{1-(s_n+1)/p^n}) \neq 0,
\]
as required.
\end{proof}

\begin{theorem}[Theorem~\ref{intro:order-to-p-pure}]\label{order-to-purity}
Let \((A,\m)\) be a regular local ring with a finite Frobenius lift \(\phi\) and \(p \in \m\).
Fix the \(\delta\)-ring structure on \(A\) induced by \(\phi\).
Let \(f \in A\) be such that \(p,f\)form a regular sequence.
Let $\bs(f)=(s_i)_{i \geq 0}$ be the splitting-order sequence of $f$.
Assume that \(s_i \le p-1\) for every \(i\ge0\).
Then:
\begin{enumerate}
  \item \(A/f\) is perfectoid pure.
  \item \(\displaystyle \ppt(A/f,p)=\sum_{i\ge1}\frac{p-1-s_i}{p^i}\).
\end{enumerate}
\end{theorem}

\begin{proof}
We put
\[
\alpha_0:=\sum_{i\ge1}\frac{p-1-s_i}{p^i}.
\]
Then $\alpha_0 \leq t_{n-1}+\frac{p-1}{p^n}$ for every $n \geq 1$, where 
\[
t_{n-1}:=\frac{p-1-s_1}{p}+\cdots+\frac{p-1-s_{n-1}}{p^{n-1}}.
\]
We take $\alpha \in \Z[\frac{1}{p}]_{\geq 0}$ such that $\alpha \leq \alpha_0$.
By \Cref{sos-vs-polygon}, the pair
\[
(A,f^{1-\frac{s_n+1}{p^n}}p^{\alpha_0})
\]
is perfectoid pure.
By the proof of \cite{p-pure}*{Proposition~6.5}, the pair $(A/f,p^{\alpha_0})$ is perfectoid pure. Thus we obtain (1) and
\[
\ppt(A/f,p) \geq \alpha.
\]

To show the converse inequality, we take $\alpha \in \Z[\frac{1}{p}]_{\geq 0}$ such that $\alpha > \alpha_0$.
Since 
\[
\lim_{n \to \infty} (t_{n-1}+\frac{p-1}{p^n})=\alpha_0,
\]
there exists an integer $n \geq 1$ such that $t_{n-1}+\frac{p-1}{p^n} < \alpha$.
Since the pair
\[
(A,f^{1-\frac{s_n}{p^{n}}}p^\alpha)
\]
is not perfectoid pure by \cref{sos-vs-polygon}.
By the proof of \cite{p-pure}*{Theorem~6.6}, the pair $(A/f,p^{\alpha})$ is not perfectoid pure, as desired.
\end{proof}

\begin{question}\label{ques-converse}
Does the converse direction of \cref{order-to-purity} hold?
That is, let \((A, \m)\) be a regular local ring admitting a finite Frobenius lift \(\phi\), and suppose \(p \in \m\).
Fix the \(\delta\)-ring structure on \(A\) induced by \(\phi\).
Let \(f \in A\) be such that \(p, f\) form a regular sequence.
If \(A/f\) is perfectoid pure, do we have \(s_i \le p - 1\) for all \(i \ge 1\)?
\end{question}

\begin{proposition}\label{explicit-n}
Let $(A,\m)$ be a regular local ring with $p \in \m$ and admitting a finite Frobenius lift $\phi$. 
Let $f \in A$ be such that $f,p$ form a regular sequence.
Let $\bs(f)=(s_i)_{i \geq 0}$ be the splitting-order sequence of $f$ and let $n \geq 1$ be an integer.
Set
\[
t_{n-1}:=\frac{p-1-s_1}{p}+\cdots+\frac{p-1-s_{n-1}}{p^{n-1}}.
\]
Then the pair
\[
(A,f^{1-\frac{1}{p^n}}p^{t_{n-1}+\frac{p-s_n}{p^n}})
\]
is not perfectoid pure.
\end{proposition}

\begin{proof}
Let $\tau$, $(B,(d))$, and $\tau'$ be as in the proof of \cref{sos-vs-polygon}.
It suffices to show that
\[
\sPhi{B,d}{s_0,\ldots,s_{n-1},s_n}
(\tau' y^{1-1/p^n}) = 0
\]
by \cref{perfectoid-Q}.

Choose $x_1,\ldots,x_d$ and $a$ as in the proof of \cref{sos-vs-polygon}.
Then there exists 
\[
\alpha\in
\Ker\!\Bigl(
\var{A}\oplus F_*\var{A}\oplus\cdots\oplus F^n_*\var{A}
\to \sQ{A,f}{0,s_1,\ldots,s_n}
\Bigr)
\]
such that
\[
(a,0,\ldots,0)+\alpha
\equiv 0
\pmod{(x_1,\ldots,x_d)(\var{A}\oplus\cdots\oplus F^n_*\var{A})}.
\]

By \cref{f-versus-y-var}, we obtain
\[
(a,0,\ldots,0)y^{1-1/p^n}
\equiv 0
\pmod{(x_1,\ldots,x_d,f)\,\sQ{B,d}{s_0,\ldots,s_{n-1},s_n}}.
\]
Therefore,
\[
\sPhi{B,d}{s_0,\ldots,s_{n-1},s_n}
(\tau' y^{1-1/p^n}) = 0,
\]
as required.
\end{proof}

\begin{theorem}[Theorem~\ref{intro:p-pure-to-order}]\label{case:p-2/p-1}
Let \((A,\m)\) be a regular local ring with a Frobenius lift \(\phi\) and \(p \in \m\).
Fix the \(\delta\)-ring structure on \(A\) induced by \(\phi\).
Let \(f \in A\) be such that \(p,f\) form a regular sequence.
Let \(\bs(f) = (s_i)_{i \ge 0}\) be the splitting-order sequence of \(f\).
\begin{enumerate}
    \item If \(A/f\) is perfectoid pure with \(\ppt(A/f,p) = \frac{p-2}{p-1}\), then \(s_n = 1\) for every \(n \ge 1\).
    \item If $s_1=\cdots =s_r=p-1$ and $s_{r+1}=p$ for some $r \geq 2$, then $A/f$ is not perfectoid pure.
    \item If $A/f$ is regular, then $s_n \leq p-1$ for every $n \geq 0$.
\end{enumerate}
In particular, if \(p=2\) or \(A/f\) is regular, 
then perfectoid purity and \(\ppt(A/f,p)\) are invariant modulo \(p^2\); 
that is, for any \(f' \in A\) with \(f \equiv f' \pmod{p^2}\), 
we have that \(A/f\) is perfectoid pure if and only if so is \(A/f'\), 
and moreover \(\ppt(A/f,p) = \ppt(A/f',p)\).
\end{theorem}

\begin{proof}
The final assertion follows immediately from (1), (3), and \cref{rmk:modulo-p^2}.

First, we prove (1).
We use the same notation as in the proof of \cref{order-to-purity},
and prove the claim by induction on \(n\).

We note that
\[
\frac{p-2}{p-1}
  = \sum_{i \ge 1} \frac{p-2}{p^i}.
\]
Assume that \(s_i=1\) for $n-1 \geq i \geq 1$.
By \cref{order-to-purity} and \cref{explicit-n}, we have
\[
\ppt(A/f,p)
 = \frac{p-2}{p-1}
 \le \frac{p-2}{p} + \cdots + \frac{p-2}{p^{n-1}} + \frac{p - s_n}{p^n}.
\]
This inequality implies \(s_n \le 1\).

If \(s_n = 0\), then \(A/f\) is quasi-\(F\)-split of height \(n\) by \cref{qfs-splitting-order},
and thus \(\ppt(A/f,p) > \frac{p-2}{p-1}\) by \cite{Yoshikawa25}*{Theorem~A},
a contradiction.
Therefore \(s_n = 1\), as desired.

Next, we prove (2).
By \cref{explicit-n}, we have
\[
\tau'f^{1-\frac{1}{p^n}}=0.
\]
Thus, the pair $(A,f^{1-\frac{1}{p^n}})$ is not perfectoid pure.
Thus, $A/f$ is not perfectoid pure by \cite{p-pure}*{Theorem~6.6}.

Finally, we prove (3).
If $\var{A/f}$ is regular, then $s_n = 0$ for every $n \ge 0$ by \cref{qfs-splitting-order}.
Hence we assume that $\var{A/f}$ is not regular; in particular, $f \in (\m^2, p)$.
Write $f \equiv pv \pmod{\m^2}$.
Since $f \notin \m^2$, we have $v \notin \m$.
It follows that
\[
\Delta(f) \equiv v^p \pmod{(p,\m)},
\]
so $\Delta(f)$ is a unit.
Therefore $f$ is a distinguished element in $B$, and in particular we have $fB = dB$.

Since the map $A/f \to B/d$ is pure, it follows that 
$A/f^s \to B/d^s$ is pure for every $s \ge 1$.
By the exact sequence in \cref{exact-seq}, the induced map
\[
\sQ{A,f}{s_0,\ldots,s_n+1}
  \longrightarrow
\sQ{B,d}{s_0,\ldots,s_n+1}
\]
is pure for all $n \ge 1$.

\smallskip
For $n \ge 1$, set
\[
\alpha_n := \frac{p-1-s_1}{p} + \cdots + \frac{p-1-s_n}{p^n}.
\]
We prove by induction on $n \ge 1$ that
\[
p^{\alpha_n}\eta' \ne 0,
\qquad
p^{\alpha_n + 1/p^n}\eta' = 0,
\qquad\text{and}\qquad s_n \le p-1.
\]

\smallskip
Assume $s_i \le p-1$ for all $i \le n-1$.
By the definition of $s_n$, we have
\[
\sPhi{A,f}{s_0,\ldots,s_n}(\eta) = 0,
\]
and hence 
\[
\sPhi{B,d}{s_0,\ldots,s_n}(\eta') = 0.
\]
By \cref{perfectoid-Q}, this implies 
\[
p^{\alpha_n + 1/p^n}\eta' = 0.
\]

If $n = 1$, the perfectoid purity of $A/f$ gives $s_1 \le p-1$.
If $n \ge 2$ and $s_n = p$, then 
\[
p^{\alpha_{n-1}}\eta' = 0,
\]
contradicting the induction hypothesis.
Therefore $s_n \le p-1$, and moreover
\[
\sPhi{A,f}{s_0,\ldots,s_n+1}(\eta) \ne 0.
\]

Since the map
\[
\sQ{A,f}{s_0,\ldots,s_n+1}
  \longrightarrow
\sQ{B,d}{s_0,\ldots,s_n+1}
\]
is pure, we obtain
\[
\sPhi{B,d}{s_0,\ldots,s_n+1}(\eta') \ne 0.
\]
By \cref{perfectoid-Q}, it follows that
\[
p^{\alpha_n}\eta' \ne 0,
\]
as desired.
\end{proof}

\begin{proof}[Proof of Theorem~\ref{intro:p-pure-to-order}]
Assertions \textup{(2)} and \textup{(3)} follow from \cref{case:p-2/p-1}.
Thus, we may assume that \(A/f\) is perfectoid pure with 
\(\ppt(A/f,p) \ge \frac{p-2}{p-1}\).
If equality holds, the claim follows from 
\cref{case:p-2/p-1}\textup{(1)}.
If instead \(\ppt(A/f,p) > \frac{p-2}{p-1}\),
then \(A/f\) is quasi-\(F\)-split by
\cite{Yoshikawa25}*{Theorem~A}.
By \cref{qfs-splitting-order}, we obtain \(s_n \le 1\)
for every \(n \ge 1\), as desired.
\end{proof}

The theory developed for computing the perfectoid pure threshold 
gives rise to the following interesting example.
The proposition below arose from discussions with Linquan Ma.

\begin{proposition}[cf.~\cref{example3}]\label{ideal-different}
Let $p=2$. 
Let \((A,\m)\) be an unramified complete regular local ring such that \(0 \ne p \in \m\) and $A/\m$ is perfect.
Let \(f \in A\) be such that \(p,f\) form a regular sequence, and set \(R:=A/f\).
Assume that \(R\) is perfectoid pure with \(\ppt(R,p)=0\).
Let \(A \to A_\infty\) be a faithfully flat extension to a perfectoid ring 
containing compatible systems of $p$-power roots of both \(f\) and \(p\).
If \(\var{R}\) is normal, then
\[
\bigcap_{n \ge 1} (f^{1/p^\infty},p^{1-1/p^n})A_\infty 
   \;\ne\; (f^{1/p^\infty},p)A_\infty.
\]
\end{proposition}

\begin{proof}
Take a Frobenius lift $\phi$ on $A$ and endow $A$ with the corresponding $\delta$-ring structure.
Let $(B,(d))$ be the perfect prism associated to $A_\infty$.
By replacing $A_\infty$ with a $p$-completely faithfully flat extension if necessary,
we may assume there exists a morphism of $\delta$-rings $A \to B$ (see the proof of \cref{good-perfd-cover}).

Since $A_\infty$ contains a compatible system of $p$-power roots of $f$,
there exists $y \in B$ such that $y-f \in (d)$ and $\phi(y)=y^p$.
Set $R_\infty:=A_\infty/f^{1/p^\infty}$ and $S:=B/(y^{1/p^\infty})$.
Then $(S,(d))$ is a perfect prism satisfying $S/(d)\simeq R_\infty$,
and by the proof of \cref{order-to-purity}, the map $R \to R_\infty$ is pure.

Since $\ppt(R,p)=0$ and $p=2$, the splitting-order sequence is $\bs(f)=(0,1,1,\ldots)$.
By \cref{perfectoid-Q}, each map $\sPhi{S,d}{s_0,\ldots,s_n}$ is surjective and
\[
\Ker(\sPhi{S,d}{s_0,\ldots,s_n}) = p^{1-1/p^n}R_\infty.
\]
Suppose, for contradiction, that
\[
\bigcap_{n \ge 1} (f^{1/p^\infty},p^{1-1/p^n})A_\infty = (f^{1/p^\infty},p)A_\infty.
\]
Then the induced map
\[
\sPhi{S,d}{\bs(f)}\colon 
   \var{R_\infty} \longrightarrow \varprojlim_{n} \sQ{S,d}{s_0,\ldots,s_n}
\]
is an isomorphism.
By functoriality (\cref{functoriality}) and the purity of $R \to R_\infty$,
the map
\[
(\sPhi{f}{s_0,\ldots,s_n})_{n \ge 1}\colon 
   \var{R} \longrightarrow \varprojlim_{n} \sQ{f}{s_0,\ldots,s_n}
\]
is pure.
Since $R$ is complete, purity implies that the map $\sPhi{f}{\infty}$ splits as an $\var{R}$-module homomorphism.
Because the isomorphism $\sQ{f}{s_0,\ldots,s_n} \simeq \var{W}_n(R)$ is compatible with the restriction maps,
the map
\[
(\Phi_{R,n})_{n \ge 1}\colon 
   \var{R} \longrightarrow \varprojlim_{n} Q_{R,n}
\]
also splits as $\var{R}$-modules by \cref{qfs-splitting-order}.

Let $U$ denote the regular locus of $\Spec \var{R}$.
Then $U$ is smooth since $A/\m$ is perfect.
By \cite{Petrov}*{Corollary~A.2}, 
since $Q_{U,n}$ is flat over $U$ (see \cite{OTY}*{Definition~3.1}),
there exists an integer $n>0$ such that $U \to Q_{U,n}$ splits.
Because $\var{R}$ is normal, the map $\var{R} \to Q_{R,n}$ also splits;
in particular, $R$ is $n$-quasi-$F$-split.
This contradicts the splitting-order sequence $(0,1,1,\ldots)$.

Hence the assumed equality cannot hold, and we obtain
\[
\bigcap_{n \ge 1} (f^{1/p^\infty},p^{1-1/p^n})A_\infty 
   \ne (f^{1/p^\infty},p)A_\infty,
\]
as desired.
\end{proof}

\section{A computational method for splitting-order sequences}
In this section, we provide a computational criterion for determining the splitting-order sequence, analogous to Fedder’s criterion in $F$-singularity theory. Using this, we give an explicit connection between the perfectoid $p$-purity threshold and the $F$-pure threshold in the regular case.

\begin{notation}\label{fedder}
Let \((A,\m)\) be a regular local ring equipped with a Frobenius lift \(\phi\),
and assume that \(p \in \m\).
Fix the \(\delta\)-ring structure on \(A\) induced by \(\phi\).
Let \(f \in A\) be such that \(p,f\) form a regular sequence,
and let \(\bs(f) = (s_i)_{i \ge 0}\) denote the splitting-order sequence of \(f\).
We set $\m^{[p^e]}:=(p,x^{p^e} \mid x \in \m)$.

Fix a generator \(u \in \Hom_{\var{A}}(F_*\var{A},\var{A})\).
Note that since \(\phi\) is flat, the Frobenius \(F\) on \(\var{A}\) is also flat.

For \(n \ge 1\), set
\[
Q_n := F_*\var{A} \oplus F^2_*\var{A} \oplus \cdots \oplus F^{n}_*\var{A}.
\]
Then we have a natural isomorphism
\begin{equation}\label{eq:isom}
\Hom_A(Q_n,\var{A})
  \;\simeq\;
  \Hom_A(F^n_*\var{A},\var{A}) \oplus \cdots \oplus \Hom_A(F_*\var{A},\var{A})
  \;\overset{(\star_1)}{\simeq}\;
  F_*^n\var{A} \oplus \cdots \oplus F_*\var{A},
\end{equation}
where \((\star_1)\) is given, for each \(1 \le e \le n\), by
\[
F^e_*\var{A} \longrightarrow \Hom_A(F^e_*\var{A},\var{A}),
\qquad
F^e_*a \longmapsto
\bigl(F^e_*b \longmapsto u^e(F^e_*(ab))\bigr).
\]
The homomorphism corresponding to
\((F^n_*g_n,\ldots,F_*g_1)\)
under \eqref{eq:isom}
is denoted by \(\psi_{(g_1,\ldots,g_n)}\).

Finally, for integers
\(0 \le l_1,\ldots,l_{n-1} \le p-1\) and \(0 \le l_n \le p\),
we define ideals \(I(l_1,\ldots,l_n)\) of \(\var{A}\) inductively by
\begin{itemize}
  \item \(I(l_n) := f^{p-l_n}\var{A}\),
  \item once \(I(l_2,\ldots,l_n)\) is defined, set
  \[
    I(l_1,\ldots,l_n)
      := f^{p-l_1-1}\,u\bigl(F_*(\Delta(f)^{l_1}I(l_2,\ldots,l_n))\bigr)
         + f^{p-l_1}\var{A}.
  \]
\end{itemize}
\end{notation}

The proof of \cref{criterion}  is almost identical to the proof of \cite{Yoshikawa25}*{Theorem~4.13}. However, for the convenience of the reader, we provide a proof.

\begin{theorem}\label{criterion}
With notation as in \cref{fedder},
for \(n \ge 1\), if \(s_1,\ldots,s_{n-1} \le p-1\), then
\[
s_n
  = \max\bigl\{\,0 \le s \le p \,\bigm|\,
      I(s_1,\ldots,s_{n-1},s) \subseteq \m^{[p]}\,\bigr\}.
\]
\end{theorem}

\begin{proof}
Fix \(n \ge 1\),
\(0 \le l_1,\ldots,l_{n-1} \le p-1\), and \(0 \le l_n \le p\).
Consider the composition
\[
\Phi \colon
\var{R}
  \xrightarrow{\sPhi{f}{0,l_1,\ldots,l_n}}
  \sQ{f}{0,l_1,\ldots,l_n}
  \xrightarrow{\pi}
  F_*\sQ{f}{l_1,\ldots,l_n},
\]
where
\(\pi(a_0,\ldots,a_n)=(a_0^p+a_1,\ldots,a_n)\);
note that \(\pi\) is an isomorphism.

The map \(\Phi\) induces
\[
\Hom_{\var{R}}(\sQ{f}{l_1,\ldots,l_n},\var{R})
  \longrightarrow \var{R},
  \qquad
  \psi \longmapsto \psi \circ \Phi(1)
  = \psi(1,0,\ldots,0).
\]
It suffices to show that
\begin{equation}\label{eq:evaluation}
u(F_*I(l_1,\ldots,l_n))\,\var{R}
  = \Im\!\bigl(
      \Hom_{\var{R}}(\sQ{f}{l_1,\ldots,l_n},\var{R})
        \to \var{R}
    \bigr).
\end{equation}

 Step 1. Proof of the inclusion \((\supseteq)\) in \eqref{eq:evaluation}.
Take
\(\psi' \in \Hom_{\var{R}}(\sQ{f}{l_1,\ldots,l_n},\var{R})\).
Since there is a surjection \(Q_n \twoheadrightarrow \sQ{f}{l_1,\ldots,l_n}\)
and \(Q_n\) is a free \(\var{A}\)-module,
there exists an \(\var{A}\)-linear map
\(\psi \colon Q_n \to \var{A}\)
such that
\begin{equation}\label{eq:inclusion}
\psi(a_1f^{l_1},
     -a_1^p\Delta(f)^{l_1}+a_2f^{l_2},
     \ldots,
     -a_{n-1}^p\Delta(f)^{l_{n-1}}+a_nf^{l_n})
  \in f\var{A}
\end{equation}
and such that
\((\var{A}\to\var{R})\bigl(\psi(1,0,\ldots,0)\bigr)
  = \psi'(1,0,\ldots,0)\).

Hence it suffices to show that
\(\psi(1,0,\ldots,0) \in u(F_*I(l_1,\ldots,l_n)) + f\var{A}\).

Write \(\psi = \psi_{(g_1,\ldots,g_n)}\)
for some \(g_1,\ldots,g_n \in \var{A}\).
From \eqref{eq:inclusion}, taking \(a_0=\cdots=a_{n-1}=0\),
we have
\(u^n(F^n_*(g_n a f^{l_n})) \in f\var{A}\)
for all \(a \in \var{A}\),
hence \(g_n \in f^{p^n-l_n}\var{A}\).
Write \(g_n = h_n f^{p^n-p}\);
then \(h_n \in I(l_n)\).

Similarly, for \(1 \le i \le n-1\), we have
\[
g_i f^{l_i} - u(F_*(g_{i+1}\Delta(f)^{l_i})) \in f^{p^i}\var{A}.
\]
This implies \(g_i = f^{p^i-p}h_i\)
for some \(h_i \in I(l_i,\ldots,l_n)\),
and in particular \(g_1 = h_1 \in I(l_1,\ldots,l_n)\).
Therefore,
\[
\psi(1,0,\ldots,0)
  = u(F_*h_1)
  \in u(F_*I(l_1,\ldots,l_n)),
\]
as desired.

Step 2. Proof of the inclusion \((\subseteq)\) in \eqref{eq:evaluation}.
Take \(h_1 \in I(l_1,\ldots,l_n)\).
Then there exist elements \(h_2,\ldots,h_n\) with
\(h_{i+1} \in I(l_{i+1},\ldots,l_n)\)
such that
\[
h_i - f^{p-l_i-1}u(F_*(h_{i+1}\Delta(f)^{l_i}))
  \in f^{p-l_i}\var{A}
  \qquad(1 \le i \le n-1).
\]
Set \(g_i := f^{p^i-p}h_i\).
Then for every \(a \in \var{A}\),
\begin{align*}
&u^i(F^i_*(a g_i f^{l_i}))
  - u^{i+1}(F^{i+1}_*(a^p g_{i+1}\Delta(f)^{l_i})) \\
&\qquad
  = u^i\!\bigl(F^i_*(a(g_i f^{l_i}
      - u(F_*g_{i+1}\Delta(f)^{l_i})))\bigr) \\
&\qquad
  = u^i\!\bigl(F^i_*(a f^{p^i-p+l_i}
      (h_i - f^{p-l_i-1}u(F_*(h_{i+1}\Delta(f)^{l_i}))))\bigr)
  \in f\var{A}.
\end{align*}
Hence \(\psi = \psi_{(g_1,\ldots,g_n)}\)
induces a homomorphism
\(\psi' \colon \sQ{f}{l_1,\ldots,l_n} \to \var{R}\)
such that
\[
\psi'(1,0,\ldots,0)
  = (\var{A} \to \var{R})(u(F_*h_1))
  \in u(F_*I(l_1,\ldots,l_n))\,\var{R},
\]
which proves \((\subseteq)\) in \eqref{eq:evaluation}.
\end{proof}

\begin{corollary}\label{fedder-another}
With notation as in \cref{fedder}.
We have
\[
s_n=\max\{0 \leq s \leq p \mid f(s_1,\ldots,s_{n-1},s) \in \m^{[p^n]} \}.
\]
Here we define $f(l):=f^{p-l}$ and, inductively,
\[
f(l_1,\ldots,l_n):=(f(l_1+1)^p\Delta(f)^{l_1})^{p^{n-2}}f(l_2,\ldots,l_n).
\]
\end{corollary}

\begin{proof}
We first prove that
\[
I(l_1,\ldots,l_i)
=
\sum_{j=1}^i u^{j-1}\bigl(F^{j-1}_*(f(l_1,\ldots,l_j)\var{A})\bigr)
\]
for $i \geq 1$ by induction on $i$.
For $i=1$, the assertion is clear.
Assume that $i \geq 2$.
Then we have
\begin{align*}
I(l_1,\ldots,l_i)
&=f(l_1+1)u\bigl(F_*(\Delta(f)^{l_1}I(l_2,\ldots,l_i))\bigr)+f(l_1)\var{A} \\
&=\sum_{j=2}^{i} u^{j-1}\Bigl(
F^{j-1}_*\bigl((f(l_1+1)^p\Delta(f)^{l_1})^{p^{j-2}}f(l_2,\ldots,l_j)\bigr)
\Bigr)+f(l_1)\var{A} \\
&=\sum_{j=2}^{i} u^{j-1}\bigl(F^{j-1}_*(f(l_1,\ldots,l_j))\bigr)+f(l_1)\var{A} \\
&=\sum_{j=1}^{i} u^{j-1}\bigl(F^{j-1}_*(f(l_1,\ldots,l_j)\var{A})\bigr).
\end{align*}

Next, we prove the assertion by induction on $n \geq 1$.
If $n=1$, the statement follows from \Cref{criterion}.
Assume that $n \geq 2$.
By the induction hypothesis, we have
$f(s_1,\ldots,s_i) \in \m^{[p^i]}$ for all $i \leq n-1$.
Hence,
\[
I(s_1,\ldots,s_{n-1},s)
\equiv
u^{n-1}\bigl(F^{n-1}_*(f(s_1,\ldots,s_{n-1},s)\var{A})\bigr) \pmod{\m^{[p]}}.
\]
Therefore, the assertion follows from \Cref{criterion}, as desired.
\end{proof}

\begin{theorem}\label{regular-case}
Let \((A,\m)\) be a regular local ring equipped with a Frobenius lift \(\phi\),
and assume that \(p \in \m\).
Fix the \(\delta\)-ring structure on \(A\) induced by \(\phi\).
Let \(f \in A\) be such that \(p,f\) form a regular sequence.
If \(A/f\) is regular, then
\[
\ppt(A/f,p)=\fpt(\var{A},\var{f}),
\]
where \(\var{f}\) denotes the image of \(f\) under \(A \to \var{A}\).
\end{theorem}

\begin{proof}
First, we assume $\var{A/f}$ is regular. Then $\var{A/f}$ is $F$-pure.
Thus, we have
\[
\ppt(A/f,p)=1=\fpt(\var{A},\var{f}).
\]
Therefore, we may assume $\var{A/f}$ is not regular, and in particular, $\Delta(f)$ is a unit by the argument in \cref{case:p-2/p-1}(3).
Define
\[
\nu_f(p^e)
  := \max\{\,N \mid \var{f}^N \notin \m^{[p^e]}\,\},
\]
so that
\(\fpt(\var{A},\var{f})=\lim_{e \to \infty} \nu_f(p^e)/p^e.\)

Let \(\bs(f) = (s_i)_{i \ge 0}\) be the splitting-order sequence of \(f\).
By \cref{case:p-2/p-1}(3), we have \(s_i \le p-1\) for all \(i \ge 0\).

Step 1.  Description of \(I(l_1,\ldots,l_n)\).
We claim that for \(0 \le l_1,\ldots,l_n \le p-1\),
\begin{equation}\label{eq:ideal}
I(l_1,\ldots,l_n)
  = u^{n-1}\!\bigl(F^{n-1}_*(f^{p^n-p^{n-1}l_1-p^{n-2}l_2-\cdots-l_n}\var{A})\bigr).
\end{equation}
We prove this by induction on \(n\).

If \(n=1\), the formula is clear.
Assume \(n \ge 2\) and the statement holds for \(n-1\).
Since \(\Delta(f)\) is a unit,
we have:
\begin{align*}
I(l_1,\ldots,l_n)
  &= f^{p-l_1-1}\,u(F_*I(l_2,\ldots,l_n)) + f^{p-l_1}\var{A} \\
  &\overset{(\star_1)}{=}
     u^{n-1}\!\bigl(F^{n-1}_*(f^{p^n-p^{n-1}l_1-\cdots-l_n}\var{A})\bigr)
     + f^{p-l_1}\var{A} \\
  &\overset{(\star_2)}{=}
     u^{n-1}\!\bigl(F^{n-1}_*(f^{p^n-p^{n-1}l_1-\cdots-l_n}\var{A})\bigr),
\end{align*}
where \((\star_1)\) follows from the induction hypothesis,
and \((\star_2)\) follows from
\[
u^{n-1}\!\bigl(F^{n-1}_*(f^{p^n-p^{n-1}l_1-\cdots-l_n}\var{A})\bigr)
  \supseteq u^{n-1}\!\bigl(F^{n-1}_*(f^{p^n-p^{n-1}l_1}\var{A})\bigr)= f^{p-l_1}\var{A}.
\]
This proves \eqref{eq:ideal}.

Step 2.  Relation between \(\nu_f(p^n)\) and \(\bs(f)\).
We now show that
\[
\frac{\nu_f(p^n)}{p^n}
  = \sum_{1 \le i \le n} \frac{p-1-s_i}{p^i}
  =: \alpha_n.
\]
By \eqref{eq:ideal} and \cref{criterion}, we have
\[
\var{f}^{\,p^n-p^{n-1}s_1-\cdots-s_n} \in \m^{[p^n]},
\qquad
\var{f}^{\,p^n-p^{n-1}s_1-\cdots-s_n-1} \notin \m^{[p^n]}.
\]
Hence
\[
\frac{\nu_f(p^n)}{p^n}
  = \frac{p^n-p^{n-1}s_1-\cdots-s_n-1}{p^n}
  = \alpha_n.
\]

Taking the limit as \(n \to \infty\), we obtain
\[
\fpt(\var{A},\var{f})
  = \lim_{n\to\infty} \alpha_n
  = \ppt(A/f,p),
\]
where the last equality follows from \cref{order-to-purity}.
\end{proof}

\begin{corollary}[Theorem~\ref{intro:rationality}]\label{rationality}
Let $d$ be a non-negative integer.
Then the following two sets coincide:
\begin{itemize}
    \item the set $\mathcal{P}_{d,p}$ of $\ppt(R,p)$ such that $(R,\m)$ is a regular local ring of dimension $d$ with $p \in \m$ and $R/pR$ is $F$-finite, and
    \item the set $\mathcal{T}_{d,p}$ of $\fpt(R,f)$ such that $(R,\m)$ is a regular $F$-finite local ring of dimension $d$ of characteristic $p$, and $f \in \m$.
\end{itemize}
In particular, the set $\mathcal{P}_{d,p}$ is contained in $\Q$ and satisfies the ascending chain condition.
\end{corollary}

\begin{proof}
First, we prove the inclusion $\mathcal{T}_{d,p} \subseteq \mathcal{P}_{d,p}$.
Thus, take a regular local ring $(R,\m)$ of dimension $d$ with $p \in \m$ such that $R/pR$ is $F$-finite.
By the proof of \cite{p-pure}*{Lemma~4.8}, we may assume that $R$ is complete.
Let $R/\m = k$ be the residue field, and let $(C,(p))$ be a complete discrete valuation ring with residue field $k$.
Then, by \cite{Matsumura}*{Theorem~29.8}, there exist a formal power series ring $A$ over $C$ and an element $f \in A$ such that $A/f \simeq R$.
Since the residue field of $C$ is $F$-finite, the ring $A$ admits a finite Frobenius lift.
Hence, by \cref{regular-case}, we obtain
\[
\ppt(R,p) = \fpt(\var{A},\var{f}) \in \mathcal{T}_{d,p}.
\]

Next, we prove the converse inclusion.
Thus, take a regular $F$-finite local ring $(R,\m)$ of dimension $d$ of characteristic $p$, and an element $f \in \m$.
We may assume that $R$ is complete, so $R$ is a formal power series ring over $k := R/\m$.
Let $(C,(p))$ be a complete discrete valuation ring with residue field $k$, and set $\widetilde{R} := C[[x_1,\ldots,x_d]]$.
Write
\[
f = \sum_{i=0}^{\infty} a_i M_i \in \m
\]
for the monomial decomposition with coefficients $a_i \in k$.
Set
\[
\widetilde{f} := \sum_{i=0}^{\infty} \widetilde{a}_i M_i + p,
\]
where each $\widetilde{a}_i \in C$ is a lift of $a_i$.
Then $\widetilde{f}$ is a lift of $f$ and satisfies $\widetilde{f} \notin \n^{2}$, where $\n$ is the maximal ideal of $\widetilde{R}$.
In particular, the quotient ring $\widetilde{R}/(\widetilde{f})$ is regular.
Thus, by \cref{regular-case}, we obtain
\[
\fpt(R,f) = \ppt(\widetilde{R}/(\widetilde{f}),p) \in \mathcal{P}_{d,p}.
\]

Therefore, we conclude that $\mathcal{P}_{d,p} = \mathcal{T}_{d,p}$.

The final assertion follows from \cite{BMS}*{Theorem~3.1} and \cite{Sato21}*{Theorem~4.7}.
\end{proof}

\begin{corollary}\label{sup-ppt-fpt}
Let $\var{A}$ be a regular local ring, $\var{f} \in \var{A}$, and $\var{R}:=\var{A}/\var{f}$.
Let $R'$ be a regular lift of $\var{R}$. Then we have
\[
\ppt(R',p)=\sup\{\ppt(R,p) \mid \text{$R$ is a lift of $\var{R}$}\}.
\]
In particular, we have
\[
\fpt(\var{A},\var{f})=\sup\{\ppt(R,p) \mid \text{$R$ is a lift of $\var{R}$}\}.
\]
\end{corollary}

\begin{proof}
The second assertion follows from the first assertion and \cref{regular-case}.
By the proof of \cite{p-pure}*{Lemma~4.9}, the perfectoid pure threshold with respect to $p$ is preserved under completion.
Thus, we may assume that $\var{A}=k[[x_1,\ldots,x_N]]$ for some field $k$ of characteristic $p$.

Let $C_k$ be the complete divisorial valuation ring with uniformizer $p$, and set $A:=C_k[[x_1,\ldots,x_N]]$.
Let $\phi$ be the Frobenius lift on $A$ defined by $\phi(x_i)=x_i^p$.
We take a lift $f_{\reg} \in A$ such that $A/f_{\reg}$ is regular.

Since every complete lift of $\var{R}$ is a hypersurface in $A$, it is sufficient to show 
\begin{equation}\label{eq:sup-ppt}
    \ppt(A/f_{\reg},p)=\sup\{\ppt(A/f,p) \mid \text{$f \in A$ with $f \equiv f_{\reg} \pmod{p}$}\}.
\end{equation}

If $\var{A/f_{\reg}}$ is regular, then $\ppt(A/f_{\reg},p)=1$, and thus \eqref{eq:sup-ppt} holds.
Therefore, we may assume that $\var{A/f_{\reg}}$ is not regular.
In particular, $\Delta(f_{\reg})$ is a unit by the proof of \Cref{case:p-2/p-1}(3).

Take $f \in A$ with $f \equiv f_{\reg} \pmod{p}$.
It is sufficient to show that
\[
\ppt(A/f_{\reg},p) \geq \ppt(A/f,p).
\]

Let $I(l_1,\ldots,l_n)$ (resp. $I^{\reg}(l_1,\ldots,l_n)$) be the ideal of $\var{A}$ defined in \cref{fedder} with respect to $f$ (resp. $f_{\reg}$).
We prove
\[
I(l_1,\ldots,l_n) \subseteq I^{\reg}(l_1,\ldots,l_n)
\]
for every $0 \leq l_1,\ldots,l_{n-1} \leq p-1$ and $0 \leq l_n \leq p$ by induction on $n$.

If $n=1$, then we have
\[
I(l_1)=f^{p-l_1}\var{A}=f_{\reg}^{p-l_1}\var{A}=I^{\reg}(l_1).
\]

If $n \geq 2$, then
\begin{align*}
    I(l_1,\ldots,l_n)
    &=f^{p-l_1-1}u(F_*(\Delta(f)^{l_1}I(l_2,\ldots,l_n)))+f^{p-l_1}\var{A} \\
    &\overset{(\star_1)}{\subseteq} f_{\reg}^{p-l_1-1}u(F_*(\Delta(f)^{l_1}I^{\reg}(l_2,\ldots,l_n)))+f_{\reg}^{p-l_1}\var{A} \\
    &\overset{(\star_2)}{\subseteq} f_{\reg}^{p-l_1-1}u(F_*(\Delta(f_{\reg})^{l_1}I^{\reg}(l_2,\ldots,l_n)))+f_{\reg}^{p-l_1}\var{A},
\end{align*}
where $(\star_1)$ follows from the induction hypothesis and $(\star_2)$ follows from the fact that $\Delta(f_{\reg})$ is a unit.

Therefore, we obtain
\[
I(l_1,\ldots,l_n) \subseteq I^{\reg}(l_1,\ldots,l_n)
\]
for every $0 \leq l_1,\ldots,l_{n-1} \leq p-1$ and $0 \leq l_n \leq p$.

By \cref{criterion}, we have $\bs(f_{\reg}) \leq \bs(f)$ with respect to the lexicographic order.
If $\bs(f_{\reg})=\bs(f)$, then we have $\ppt(A/f_{\reg},p)=\ppt(A/f,p)$ by \Cref{order-to-purity} and \Cref{case:p-2/p-1}(3).
Thus, we may assume that $\bs(f_{\reg}) < \bs(f)$.

Set $\bs(f)=(s_0,s_1,\ldots)$ and $\bs(f_{\reg})=(s_0^{\reg},s_1^{\reg},\ldots)$.
Then there exists $n \geq 1$ such that
\[
s^{\reg}_0=s_0,\ 
s_1^{\reg}=s_1,\ldots,
s_{n-1}^{\reg}=s_{n-1},
\quad \text{and} \quad
s_n^{\reg}<s_n.
\]
We note that $s_0=s_0^{\reg}=0$ by the definition of the splitting-order sequence.
By \Cref{case:p-2/p-1}(3), we have $s_1,\ldots,s_{n-1} \leq p-1$. 
By \Cref{explicit-n}, the pair
\[
(A,f^{1-\frac{1}{p^n}}p^{t_{n-1}+\frac{p-s_n}{p^n}})
\]
is not perfectoid pure, where 
\[
t_{n-1}
=\frac{p-1-s_1}{p}
+\cdots+
\frac{p-1-s_{n-1}}{p^{n-1}}
=\frac{p-1-s^{\reg}_1}{p}
+\cdots+
\frac{p-1-s^{\reg}_{n-1}}{p^{n-1}}.
\]
Then the pair $(A/f,p^{t_{n-1}+\frac{p-s_n}{p^n}})$ is not perfectoid pure by the proof of \cite{p-pure}*{Proposition~6.5}.
In particular, we have
\[
\ppt(A/f,p) \leq t_{n-1}+\frac{p-s_n}{p^n}.
\]

On the other hand, by \Cref{order-to-purity}, we have
\[
\ppt(A/f_{\reg},p)=\sum_{i \geq 1} \frac{p-1-s^{\reg}_i}{p^i}
\geq
t_{n-1}+\frac{p-1-s_n^{\reg}}{p^n}
\geq
t_{n-1}+\frac{p-s_n}{p^n}.
\]
Thus it follows that
\[
\ppt(A/f_{\reg},p) \geq \ppt(A/f,p),
\]
as desired.
\end{proof}

\section{Splitting-order sequence for graded rings}

In this section, we investigate perfectoid purity for graded rings and its relation to the local case at the homogeneous maximal ideal.  
We begin by proving that the splitting-order sequence behaves compatibly with localization (Proposition~\ref{graded-versus-local}).  
We then study graded analogues of Calabi--Yau and Fano  cases (Proposition~\ref{CY-case} and Theorem~\ref{Fano-case}),  
where we establish periodicity and eventual zeros in the splitting-order sequence,  
leading to results on perfectoid purity and BCM-regularity.

\begin{notation}\label{notation:graded}
Let $C$ be a complete divisorial valuation ring with uniformizer $p$ such that $C/(p)$ is $F$-finite.
Fix a Frobenius lift $\phi_C$ on $C$.
Let $A := C[x_1,\ldots,x_N]$, endowed with a Frobenius lift $\phi$ extending $\phi_C$ and satisfying $\phi(x_i) = x_i^p$ for $1 \le i \le N$.
We give $A$ a graded structure by setting $A_0 = C$ and $\deg(x_i) = \mu_i \in \Z_{>0}$ for $1 \le i \le N$.
Let $f \in A$ be a homogeneous element such that $(p,f)$ forms a regular sequence, and put $R := (A/f)^{\wedge p}$.
Set $\m := (p, x_1, \ldots, x_N) = (p, A_+)$ and $d := \dim(\var{A/f}) = N - 1$.

Since $\Delta(f) \in A$ is homogeneous of degree $p\deg(f)$, the module $\sQ{A,f}{l_1,\ldots,l_n}$ has a natural graded structure, and the map 
$\sPhi{A,f}{l_1,\ldots,l_n}$ is graded for all $0 \le l_1,\ldots,l_{n-1} \le p-1$ and $1 \le l_n \le p$.
We denote by $\bs(f) = (s_0, s_1, \ldots)$ the splitting-order sequence of $f$.
\end{notation}

\begin{proposition}\label{graded-versus-local}
With notation as in \cref{notation:graded}, let $\bs(f/1)$ denote the splitting-order sequence of $f/1 \in A_\m$, 
where the $\delta$-structure on $A_\m$ is induced by $\phi_\m$.
Then $\bs(f) = \bs(f/1)$.
In particular, if each term of $\bs(f)$ is at most $p-1$, then $R$ is perfectoid pure.
Furthermore, we have
\[
\ppt(R,p)=\sum_{i\ge1}\frac{p-1-s_i}{p^i}.
\]
\end{proposition}

\begin{proof}
Take integers $0 \le l_1, \ldots, l_{n-1} \le p-1$ and $1 \le l_n \le p$.
Since $\sQ{A,f}{l_1,\ldots,l_n}$ is a finite $\var{A}$-module, 
the purity of $\sPhi{A,f}{l_1,\ldots,l_n}$ is equivalent to the splitting of 
$\sPhi{A,f}{l_1,\ldots,l_n}$.
Moreover, as $\sPhi{A,f}{l_1,\ldots,l_n}$ is a graded $\var{A}$-module homomorphism, 
its splitting is equivalent to the splitting of its localization
\[
(\sPhi{A,f}{l_1,\ldots,l_n})_\m 
  = \sPhi{A_\m,f/1}{l_1,\ldots,l_n}
\]
as a homomorphism of $\var{A_\m}$-modules by \cite{IY}*{Lemma~2.10}.
By definition, this shows that $\bs(f) = \bs(f/1)$, as desired.

Finally, the last assertion follows from \cref{order-to-purity} together with \cite{IY}*{Corollary~4.27}.
\end{proof}

\begin{proposition}\label{CY-case}
With notation as in \cref{notation:graded}, assume $\deg(f)=\mu_1+\cdots+\mu_N$.
If $s_r=0$ for some $r>0$, then $s_{i+r}=s_i$ for every $i\ge 0$.
In particular, $R$ is perfectoid pure.
\end{proposition}

\begin{proof}
Let $\eta\in H^d_\m(\var{R})$ be a nonzero socle element. 
By the assumption $\deg(f)=\sum_i\mu_i$, the element $\eta$ is homogeneous of degree $0$.

Fix $n$ with $s_n\le p-1$.
Then there exists $\eta_n\in H^d_\m(\var{R})\setminus\{0\}$ such that
\begin{align*}
 &\bigl(H^d_\m(\var{R}) \xrightarrow{\sPhi{f}{s_0,\ldots,s_{n-1},s_n+1}} 
   H^d_\m(\sQ{f}{s_0,\ldots,s_{n-1},s_n+1})\bigr)(\eta) \\
 &=
 \bigl(F^n_*H^d_\m(\var{R})
   \xrightarrow{\cdot F^n_*f^{s_n}}
   F^n_*H^d_\m(\var{A/f^{s_n+1}})
   \xrightarrow{V}
   H^d_\m(\sQ{f}{s_0,\ldots,s_{n-1},s_n+1})\bigr)(\eta_n).
\end{align*}
In particular, we have
\begin{align}\label{eq:CY}
 \bigl(H^d_\m(\var{R}) \xrightarrow{\cdot f^{s_{n-1}}}
   H^d_\m(\var{A/f^{s_{n-1}+1}})
   \xrightarrow{\sPhi{f}{s_{n-1},s_n+1}}
   H^d_\m(\sQ{f}{s_{n-1},s_n+1})\bigr)(\eta_{n-1})
 \\[2pt]
 =
 \bigl(F_*H^d_\m(\var{R}) \xrightarrow{\cdot F_*f^{s_n}}
   F_*H^d_\m(\var{A/f^{s_n+1}})
   \xrightarrow{V}
   H^d_\m(\sQ{f}{s_{n-1},s_n+1})\bigr)(\eta_n).
 \notag
\end{align}
Moreover,
\begin{align}\label{eq:CY2}
 &\bigl(H^d_\m(\var{R}) \xrightarrow{\sPhi{f}{s_0,\ldots,s_n,s}} 
   H^d_\m(\sQ{f}{s_0,\ldots,s_n,s})\bigr)(\eta) \\
 &=
 \bigl(F^n_*H^d_\m(\var{R}) \xrightarrow{\cdot F^n_*f^{s_n}}
   F^n_*H^d_\m(\var{A/f^{s_n+1}})
   \xrightarrow{V \circ \sPhi{f}{s_n,s}}
   H^d_\m(\sQ{f}{s_0,\ldots,s_n,s})\bigr)(\eta_n).\notag
\end{align}

Since $V$ is injective, the map $\sPhi{f}{s_0,\ldots,s_n,s}$ is pure 
if and only if the element
\[
\bigl(H^d_\m(\var{R}) \xrightarrow{\cdot f^{s_n}}
   H^d_\m(\var{A/f^{s_n+1}})
   \xrightarrow{\sPhi{f}{s_n,s}}
   H^d_\m(\sQ{f}{s_n,s})\bigr)(\eta_n)
\]
is nonzero.

Since $s_r=0$, we have $s_1,\ldots,s_{r-1}\le p-1$ by the definition of $\bs(f)$.
Then $\eta_r$ satisfies
\begin{align*}
 &\bigl(H^d_\m(\var{R}) \xrightarrow{\sPhi{f}{s_0,\ldots,s_{r-1},1}} 
   H^d_\m(\sQ{f}{s_0,\ldots,s_{r-1},1})\bigr)(\eta) \\
 &=
 \bigl(F^r_*H^d_\m(\var{R}) \xrightarrow{V}
   H^d_\m(\sQ{f}{s_0,\ldots,s_{r-1},1})\bigr)(\eta_r).
\end{align*}
Since both $\sPhi{f}{s_0,\ldots,s_{r-1},1}$ and $V$ are graded, the element $\eta_r$ is homogeneous of degree $0$.
Hence there exists $a\in R^\times$ such that $\eta_0=a\eta_r$.
Then we have
\[
\bigl(H^d_\m(\var{R})
   \xrightarrow{\sPhi{f}{s_r,s}}
   H^d_\m(\sQ{f}{s_r,s})\bigr)(\eta_r)
   = a\,\bigl(H^d_\m(\var{R})
   \xrightarrow{\sPhi{f}{0,s}}
   H^d_\m(\sQ{f}{0,s})\bigr)(\eta),
\]
and hence $s_{r+1}=s_1\le p-1$.
By \eqref{eq:CY}, we further have $\eta_1=a^p\eta_{r+1}$.
Repeating this argument inductively gives $\eta_i=a^{p^i}\eta_{i+r}$ and $s_i=s_{i+r}$ for all $i\ge 0$.

Therefore $s_n\le p-1$ for all $n$.
Applying \cref{graded-versus-local} and \cref{order-to-purity}, we conclude that $R$ is perfectoid pure.
\end{proof}

\begin{theorem}\label{Fano-case}
With notation as in \cref{notation:graded}, assume that $m := \mu_1 + \cdots + \mu_N - \deg(f) > 0$.
Suppose that 
$K := \Ker(F \colon H^d_\m(\var{R}) \to H^d_\m(\var{R}))$ 
has finite length, and set
\[
t := \min\{\, l \in \Z \mid K_l \neq 0 \,\}.
\]
If there exists $r \ge 1$ such that 
$-p^r m < t$ and $s_r = 0$, 
then the following hold:
\begin{enumerate}
    \item $s_i = 0$ for every $i > r$, and in particular, $R$ is perfectoid pure.
    \item If $\Spec(\var{R}) \setminus \{\m\}$ is strongly $F$-regular, 
    then $R_\m$ is perfectoid BCM-regular.
\end{enumerate}
\end{theorem}

\begin{proof}
Let $\eta \in H^d_\m(\var{R})$ be a nonzero socle element; 
then $\eta$ is homogeneous of degree $-m$.

For each $n$ with $s_n \le p-1$, take elements 
$\eta_n \in H^d_\m(\var{R})$ as in the proof of \cref{CY-case}. 
Then each $\eta_n$ is homogeneous of degree 
$-p^n m + s_n \deg(f)$.
Since $-p^r m  = \deg(\eta_r) < t$, we have $\eta_r \notin K$.  
Therefore,
\[
\bigl(
H^d_\m(\var{R})
  \xrightarrow{\ \sPhi{f}{0,1}\ }
  H^d_\m(\sQ{f}{0,1})
  \simeq F_* H^d_\m(\var{R})
\bigr)(\eta_r)
= (F \colon H^d_\m(\var{R}) \to F_*H^d_\m(\var{R}))(\eta_r)
\neq 0.
\]
By \eqref{eq:CY2}, this implies $s_{r+1} = 0$.
Repeating the same argument inductively, we obtain (1).

\smallskip

Next, assume that $\Spec(\var{R}) \setminus \{\m\}$ is strongly $F$-regular.  
Then the tight closure of $0$ in $H^d_\m(\var{R})$, denoted by $0^*$, has finite length.
By (1), there exists $n$ such that $\eta_n \notin 0^*$ and $s_n = 0$.

Let $R \to P$ be a map to a perfectoid BCM-algebra $P$, 
and for each $i \ge 0$ set
\[
\tau_i := 
\bigl(H^d_\m(\var{R}) \longrightarrow H^d_\m(\var{P})\bigr)(\eta_i),
\quad
\text{where } \eta_0 := \eta.
\]
Since 
\[
\Ker\bigl(H^d_\m(\var{R}) \to H^d_\m(\var{P})\bigr) \subseteq 0^*,
\]
we have $\tau_n \ne 0$.

Furthermore, by the functoriality of $\Phi$ (see \cref{functoriality}), 
we obtain
\[
\sPhi{B,d}{s_0,\ldots,s_{n-1},1}(\tau_0)
  = V(F^n_*(\tau_n)) \ne 0,
\]
where $(B,(d))$ is a perfect prism satisfying $B/(d) \simeq P$.
Hence $\tau_0 \neq 0$, which implies that the map $R \to P$ is pure.
Therefore, $R_\m$ is perfectoid BCM-regular, as desired.
\end{proof}

\section{Examples}
Finally, we present explicit examples illustrating how the splitting-order sequence and perfectoid purity behave in concrete settings. These computations demonstrate the rationality and stability phenomena established in the previous sections.

\begin{notation}\label{example}
With the notation as in \cref{fedder}, we assume 
\[
A:=W(k)[[x_1,\ldots,x_N]]
\]
for a perfect field $k$ and $\phi(x_i)=x_i^p$ for $1 \leq i \leq N$.
We set $\var{x}_i:=(A \to \var{A})(x_i)$ for $1 \leq i \leq N$.
Since 
\[
\{F_*(\var{x}_1^{i_1}\cdots \var{x}_N^{i_N}) \mid 0 \leq i_1,\ldots,i_N \leq  p-1\}
\]
is a basis of $F_*\var{A}$ over $\var{A}$, we can take a generator \(u \in \Hom_{\var{A}}(F_*\var{A},\var{A})\) to be the dual basis element corresponding to $F_*(\var{x}_1^{p-1}\cdots \var{x}_N^{p-1})$.
\end{notation}

\begin{remark}
Let $f \in A' := W(k)[x_1, \ldots, x_N]$ be a homogeneous element.
Then $A/f$ is perfectoid pure if and only if $(A'/f)^{\wedge p}$ is perfectoid pure,
and moreover we have
\[
\ppt(A/f,p) = \ppt((A'/f)^{\wedge p},p)
\]
by \cref{graded-versus-local} together with the proof of \cite{p-pure}*{Lemma~4.8}.
\end{remark}

\begin{theorem}\label{Fermat-CY}
With the notation as in \cref{example}, let \(f = x_1^N + \cdots + x_N^N\) and assume \(p > N\).
For each \(e \ge 1\), take an integer \(0 \le s_e \le N-2\) satisfying \(s_e + 1 \equiv p^e \pmod{N}\).
Then the splitting-order sequence of \(f\) is \(\bs(f) = (0, s_1, s_2, \ldots)\).
In particular, \(A/f\) is perfectoid pure.
\end{theorem}

\begin{proof}
We set \(X := x_1 \cdots x_N\) and prove that
\[
I(s_1,\ldots,s_n) \subseteq \m^{[p]}
\quad \text{and} \quad
I(s_1,\ldots,s_n + 1) \nsubseteq \m^{[p]}
\]
for all \(n \ge 1\) by induction on \(n\).

For \(n = 1\), since
\[
f^{p - s_1 - 1} \equiv X^{p - s_1 - 1} \pmod{\m^{[p]}},
\]
we obtain \(f^{p - s_1} \in \m^{[p]}\), as desired.

Next, we fix \(n \ge 2\) and assume the claim holds for all smaller \(n\).
We show that
\begin{equation}\label{eq:fermat-I}
I(s_1,\ldots,s_{n-1},s)
\equiv
u^{e}\bigl(F^{e}_*(X^{p^{e+1} - p(s_e + 1)} \Delta(f)^{s_e} I(s_{e+1},\ldots,s_{n-1},s))\bigr)
\pmod{\m^{[p]}}
\end{equation}
for every \(1 \le e \le n-1\) and \(0 \le s \le p\)
by induction on \(e\).

For \(e=1\), the equality follows directly from the definition of \(I(\cdot)\) and the congruence \(f^{p - s_1 - 1} \equiv X^{p - s_1 - 1} \pmod{\m^{[p]}}\).
Suppose \(e \ge 2\), and assume that \eqref{eq:fermat-I} holds for \(e-1\).
We claim that
\begin{equation}\label{eq:fermat-delta}
\Delta(f)^{s_{e-1}} f^{p - s_e - 1}
\equiv
X^{p(s_{e-1}+1) - (s_e + 1)}
\pmod{(x_1^{pN}, \ldots, x_N^{pN})}.
\end{equation}
Indeed, since \(f = x_1^N + \cdots + x_N^N\), we have
\begin{align*}
\Delta(f)^{s_{e-1}} f^{p - s_e - 1}
&= \frac{1}{p^{s_{e-1}}} (f^p - (x_1^{pN} + \cdots + x_N^{pN}))^{s_{e-1}} f^{p - s_e - 1} \\
&\equiv \frac{1}{p^{s_{e-1}}} f^{p(s_{e-1}+1) - (s_e + 1)}
\pmod{(x_1^{pN}, \ldots, x_N^{pN})}.
\end{align*}
By the definition of \(s_e\), we have \(p(s_{e-1}+1) \equiv s_e + 1 \pmod{N}\), 
and the coefficient of \(X^{p(s_{e-1}+1)-(s_e+1)}\) in \(f^{p(s_{e-1}+1)-(s_e+1)}\) is nonzero and not divisible by \(p^{s_{e-1}+1}\). 
Thus, \eqref{eq:fermat-delta} holds.

It follows that
\begin{equation}\label{eq:fermat-X}
X^{p^{e} - p(s_{e-1}+1)} \Delta(f)^{s_{e-1}} f^{p - s_e - 1}
\equiv
X^{p^{e} - (s_e + 1)} \pmod{\m^{[p^{e}]}}
\end{equation}
and
\(X^{p^{e} - p(s_{e-1}+1)} \Delta(f)^{s_{e-1}} f^{p - s_e} \in \m^{[p^{e}]}\).
Substituting \eqref{eq:fermat-X} into the definition of \(I(\cdot)\), we obtain
\begin{align*}
I(s_1,\ldots,s_n)
&\equiv u^{e-1}\bigl(F^{e-1}_*(X^{p^{e} - p(s_{e-1}+1)} \Delta(f)^{s_{e-1}} I(s_e,\ldots,s_n))\bigr) \\
&\equiv u^{e-1}\bigl(F^{e-1}_*(X^{p^{e} - (s_e + 1)} u(F_*(\Delta(f)^{s_e} I(s_{e+1},\ldots,s_n))))\bigr) \\
&= u^{e}\bigl(F^{e}_*(X^{p^{e+1} - p(s_e + 1)} \Delta(f)^{s_e} I(s_{e+1},\ldots,s_n))\bigr)
\pmod{\m^{[p]}},
\end{align*}
which proves \eqref{eq:fermat-I} for all \(e\).

Taking \(e = n-1\) in \eqref{eq:fermat-I}, we have
\begin{align*}
I(s_1,\ldots,s_{n-1},s)
&\equiv
u^{n-1}\bigl(F^{n-1}_*(X^{p^{n} - p(s_{n-1}+1)} \Delta(f)^{s_{n-1}} f^{p-s}\var{A})\bigr) \\
&\equiv 
\begin{cases}
0 & \text{if } s = s_n,\\[3pt]
X^{p-1} & \text{if } s = s_n + 1,
\end{cases}
\pmod{\m^{[p]}},
\end{align*}
and the result follows.
\end{proof}

\begin{proposition}\label{several-criteria}
With the notation as in \cref{example}, assume that \(f \in \m^{[p]}\).
\begin{enumerate}
\item If 
\[
f^{p-1}\Delta(f)^{p-1} \equiv  v x_1^{p^2-1}\cdots x_N^{p^2-1} 
\pmod{\m^{[p^2]}}
\]
for some $v \in A^{\times}$,
then we have \(\bs(f) = (0,p-1,0,p-1,0,\ldots)\).
Hence \(A/f\) is perfectoid pure, and 
\(\ppt(A/f,p) = \frac{1}{p+1}\) by \cref{order-to-purity}.

\item If \(\Delta(f)^{p-1} \in \m^{[p^2]}\), then \(A/f\) is not perfectoid pure.

\item If 
\[
\Delta(f)^{p-1} \equiv v (x_1 \cdots x_N)^{p(p-1)} \pmod{\m^{[p^2]}}
\]
for some $v \in A^{\times}$,
then $\bs(f)=(0,p-1,p-1,\ldots)$.
Hence \(A/f\) is perfectoid pure with \(\ppt(A/f,p) = 0\) by \cref{order-to-purity}.
\end{enumerate}
\end{proposition}

\begin{proof}
\textbf{Proof of \textup{(1)}.}
We claim that
\[
J_{2m}
:= I(\underbrace{p-1,0,p-1,\ldots,p-1,1}_{2m\text{ terms}})
\equiv (x_1\cdots x_N)^{p-1}\var{A}
\pmod{\m^{[p]}}
\]
for all \(m \ge 1\), and proceed by induction on \(m\).

For \(m = 1\),
\begin{align*}
I(p-1,1)
&= u\bigl(F_*(\Delta(f)^{p-1}f^{p-1}\var{A})\bigr) + f\var{A} \\
&\equiv (x_1\cdots x_N)^{p-1}\var{A}
\pmod{\m^{[p]}}.
\end{align*}
For \(m \ge 2\),
\begin{align*}
J_{2m}
&= u\bigl(F_*(\Delta(f)^{p-1}f^{p-1}u(F_*J_{2m-2}))\bigr) + f\var{A} \\
&\equiv u\bigl(F_*((x_1\cdots x_N)^{p^2-1}u(F_*J_{2m-2}))\bigr)
\overset{(\star_1)}{\equiv} (x_1\cdots x_N)^{p-1}
\pmod{\m^{[p]}},
\end{align*}
where \((\star_1)\) follows from the induction hypothesis.

Next, we show that
\[
J_{2m-1}
:= I(\underbrace{p-1,0,p-1,\ldots,0,p-1}_{(2m-1)\text{ terms}})
\subseteq \m^{[p]}
\]
for all \(m \ge 1\).

For \(m = 1\),
\[
J_1 = I(p-1) = f\var{A} \subseteq \m^{[p]}.
\]
For \(m \ge 2\),
\begin{align*}
J_{2m-1}
&= u\bigl(F_*(\Delta(f)^{p-1}f^{p-1}u(F_*J_{2m-3}))\bigr) + f\var{A} \\
&\equiv u\bigl(F_*((x_1\cdots x_N)^{p^2-1}u(F_*J_{2m-3}))\bigr)
\overset{(\star_2)}{\equiv} 0
\pmod{\m^{[p]}},
\end{align*}
where \((\star_2)\) follows from the induction hypothesis.

Finally, we prove that
\[
J'_{2m+1}
:= I(\underbrace{p-1,0,p-1,\ldots,0,p}_{(2m+1)\text{ terms}})
\equiv (x_1\cdots x_N)^{p-1}
\pmod{\m^{[p]}}
\]
for all \(m \ge 1\).

For \(m = 1\),
\[
J'_3 = I(p-1,0,p)
= u\bigl(F_*(\Delta(f)^{p-1}f^{p-1})\bigr) + f\var{A}
\equiv (x_1\cdots x_N)^{p-1}
\pmod{\m^{[p]}}.
\]
For \(m \ge 2\),
\begin{align*}
J'_{2m+1}
&= u\bigl(F_*(\Delta(f)^{p-1}f^{p-1}u(F_*J'_{2m-1}))\bigr) + f\var{A} \\
&\equiv u\bigl(F_*((x_1\cdots x_N)^{p^2-1}u(F_*J'_{2m-1}))\bigr)
\overset{(\star_3)}{\equiv} (x_1\cdots x_N)^{p-1}
\pmod{\m^{[p]}},
\end{align*}
where \((\star_3)\) follows from the induction hypothesis.

Combining the above computations, we obtain
\[
\bs(f) = (0,p-1,0,p-1,0,\ldots),
\]
as claimed.

\bigskip
\textbf{Proof of \textup{(2)}.}
Since \(f \in \m^{[p]}\), we have \(s_1 = p-1\).
Moreover,
\[
I(p-1,p) = u(F_*\Delta(f)^{p-1}) + f\var{A} \subseteq \m^{[p]}.
\]
Hence \(s_2 = p\).
By \cref{case:p-2/p-1}(2), it follows that \(A/f\) is not perfectoid pure.

\bigskip
\textbf{Proof of \textup{(3)}.}
We claim that
\[
J_m
:= I(\underbrace{p-1,p-1,\ldots,p-1,p}_{m\text{ terms}})
\equiv (x_1\cdots x_N)^{p-1}\var{A}
\pmod{\m^{[p]}}
\]
for all \(m \ge 2\), by induction on \(m\).

For \(m = 2\),
\begin{align*}
I(p-1,p)
&= u\bigl(F_*(\Delta(f)^{p-1}\var{A})\bigr) + f\var{A} \\
&\equiv u\bigl(F_*((x_1\cdots x_N)^{p(p-1)}\var{A})\bigr)
\equiv (x_1\cdots x_N)^{p-1}\var{A}
\pmod{\m^{[p]}}.
\end{align*}
For \(m \ge 3\),
\begin{align*}
J_m
&= u\bigl(F_*(\Delta(f)^{p-1}J_{m-1})\bigr) + f\var{A} \\
&\equiv u\bigl(F_*((x_1\cdots x_N)^{p(p-1)}J_{m-1})\bigr)
\overset{(\star_4)}{\equiv} (x_1\cdots x_N)^{p-1}
\pmod{\m^{[p]}},
\end{align*}
where \((\star_4)\) follows from the induction hypothesis.

Next, we show that
\[
J'_m := I(\underbrace{p-1,p-1,\ldots,p-1}_{m\text{ terms}}) 
\subseteq \m^{[p]}
\]
for all \(m \ge 1\).

For \(m = 1\),
\[
J'_1 = I(p-1) = f\var{A} \subseteq \m^{[p]}.
\]
For \(m \ge 2\),
\begin{align*}
J'_m
&= u\bigl(F_*(\Delta(f)^{p-1}J'_{m-1})\bigr) + f\var{A} \\
&\equiv u\bigl(F_*((x_1\cdots x_N)^{p(p-1)}J'_{m-1})\bigr)
\overset{(\star_5)}{\equiv} 0
\pmod{\m^{[p]}},
\end{align*}
where \((\star_5)\) follows from the induction hypothesis.

Combining these, we obtain
\[
\bs(f) = (0,p-1,p-1,p-1,p-1,\ldots),
\]
and therefore \(A/f\) is perfectoid pure with \(\ppt(A/f,p)=0\), as claimed.
\end{proof}

\begin{example}\label{example1}
Let \(N = p\) and \(f = x_1^p + \cdots + x_p^p-px_1\cdots x_p\).
Then \(f \in \m^{[p]}\) and
\[
\Delta(f)^{p-1} = (\Delta(x_1^p+\cdots+x_p^p) + (x_1 \cdots x_p)^p)^{p-1} \equiv v (x_1 \cdots x_p)^{p(p-1)} \pmod{\m^{[p^2]}},
\]
where
\[
v=\sum^{p-1}_{i=0} \frac{1}{p^i}\binom{p-1}{i}\binom{pi}{i \cdots i} \in A^{\times}.
\]
Hence, by \cref{several-criteria}\textup{(3)}, the quotient \(A/f\) is perfectoid pure, and 
\[
\ppt(A/f,p) = 0.
\]
\end{example}

\begin{example}\label{example2}
Let $N=p+1$, and $f=x_1^{p+1}+\cdots+x_{p+1}^{p+1}$.
Then $f \in \m^{[p]}$ and
\[
f^{p-1}\Delta(f)^{p-1} \equiv (x_1\cdots x_{p+1})^{p^2-1} \quad \pmod{\m^{[p^2]}}. 
\]
Thus, we have $\bs(f)=(0,p-1,0,p-1,0,\ldots)$ and $A/f$ is perfectoid pure with $\ppt(A/f,p)=\frac{1}{p+1}$ by \cref{several-criteria}(1).
\end{example}

\begin{example}[cf.~\cref{ideal-different}]\label{example3}
Let \(p = 2\).
\begin{enumerate}
    \item Let \(N = 5\) and \(f = x_1^5 + \cdots + x_5^5\).
    Then \(A/f\) is not perfectoid pure, while 
    \(A/(f + 2x_1 \cdots x_5)\) is perfectoid pure with 
    \(\ppt(A/(f + 2x_1 \cdots x_5),p) = 0\) by \cref{several-criteria}\textup{(2)} and \textup{(3)}.
    Indeed, \(f \in \m^{[p]}\) and \(\Delta(f) \in \m^{[p^2]}\).

    \item Let \(N = 4\) and
    \[
    f = x_1^4 + x_2^4 + x_3^4 + x_4^4
      + x_1^2x_2^2 + x_1^2x_3^2 + x_2^2x_3^2
      + x_1x_2x_3(x_1 + x_2 + x_3).
    \]
    Then \(A/f\) is not perfectoid pure, whereas
    \(A/(f + 2x_1 \cdots x_4)\) is perfectoid pure with 
    \(\ppt(A/(f + 2x_1 \cdots x_4),p) = 0\) \cref{several-criteria} \textup{(2)} and \textup{(3)}.
    Indeed, \(f \in \m^{[p]}\) and \(\Delta(f) \in \m^{[p^2]}\).
\end{enumerate}
\end{example}

\begin{example}[Fermat K3 surface]\label{Fermat-K3-surface}
Let \(N = 4\) and \(f = x^4 + y^4 + z^4 + w^4\), where \(x := x_1\), \(y := x_2\), \(z := x_3\), and \(w := x_4\).
\begin{enumerate}
  \item If \(p = 2\), then \(f \in \m^{[p]}\) and \(\Delta(f) \in \m^{[p^2]}\).
  Hence \(A/f\) is not perfectoid pure, while \(A/(f + 2xyzw)\) is perfectoid pure by \cref{several-criteria}\,(3).
  
  \item If \(p \equiv 1 \pmod{4}\), then \(A/f\) is perfectoid pure with \(\ppt(A/f,p) = 1\),
  since \(\var{A/f}\) is \(F\)-pure.
  
  \item If \(p \equiv 3 \pmod{4}\), then \(A/f\) is perfectoid pure with \(\ppt(A/f,p) = \frac{2}{p^2 - 1}\).
  Indeed, for \(p = 3\), this follows from \cref{example2}.
  If \(p > 3\), then by \cref{Fermat-CY} we have \(\bs(f) = (0, 2, 0, 2, \ldots)\),
  and thus \(A/f\) is perfectoid pure with the desired value \(\ppt(A/f,p) = \frac{2}{p^2 - 1}\).
\end{enumerate}
\end{example}

\begin{example}
Let \(N = 5\) and \(f = x_1^4 + \cdots + x_5^4\).
\begin{itemize}
    \item If \(p = 2\), then \(f \in \m^{[2]}\) and \(\Delta(f) \in \m^{[4]}\).
    Hence \(R\) is not perfectoid pure by \cref{several-criteria}~(2).

    \item If \(p \equiv 1 \pmod{4}\) or \(p > 7\), then \(R\) is perfectoid pure with \(\ppt(R,p) = 1\).
    Indeed, \(f^{p-1}\) contains either \((x_1 \cdots x_4)^{p-1}\) or \((x_1 \cdots x_4)^{p-3}x_5^8\),
    both of which are not contained in \(\m^{[p]}\).

    \item If \(p = 3\), then \(\bs(f) = (0,2,0,0,\ldots)\).
    In particular, \(R\) is perfectoid pure with \(\ppt(R,p) = 1/3\).
    Since \(u(F_* f^{p-1}) = \m^2\), we have \(m = 1\) and \(t = -2\), where \(m, t\) are defined in \cref{Fano-case}.
    Hence it suffices to show that \(s_1 = 2\) and \(s_2 = 0\) by \cref{Fano-case}.
    Since \(f \in \m^{[3]}\), we have \(s_1 = 2\).
    Moreover,
    \[
\begin{aligned}
p^2 f^{p^2 - 2p - 1} \Delta(f)^2
&= f^{p^2 - 2p - 1}
   \bigl(f^p - (x_1^{4p} + \cdots + x_5^{4p})\bigr)^2 \\
&\equiv f^{p^2 - 1}
   \pmod{(x_1^9, \ldots, x_5^9)}.
\end{aligned}
\]
    The $p$-adic order of the coefficient of \((x_1 \cdots x_4)^{p^2 - 1}\) in \(f^{p^2 - 1}\) is \(2\);
    hence \(f^{p^2 - p - 1}\Delta(f)\) is not contained in \(\m^{[p^2]}\).
    Thus \(I(2,1) \nsubseteq \m^{[p^2]}\), as desired.

    \item If \(p = 7\), then \(\bs(f) = (0,1,0,0,\ldots)\).
    In particular, \(R\) is perfectoid pure with \(\ppt(R,p) = 6/7\).
    Since \(u(F_* f^{p-1}) = \m\), we have \(m = 1\) and \(t = -1\), where \(m, t\) are defined in \cref{Fano-case}.
    Hence it suffices to show that \(s_1 = 1\) and \(s_2 = 0\) by \cref{Fano-case}.
    Since \(f^{p-1} \in \m^{[p]}\) and
    \[
    f^{p-2} \equiv (x_1 \cdots x_5)^4 \notin \m^{[p]} \pmod{\m^{[p]}},
    \]
    we obtain \(s_1 = 1\).
    By the same argument as in \textup{(3)}, it suffices to show that
    the $p$-adic order of the coefficient of
    \(x_1^{48}x_2^{48}x_3^{40}x_4^{28}x_5^{28}\) in \(f^{p^2 - 1}\) is \(1\).
    Indeed, this coefficient is
    \[
    \frac{(p^2 - 1)!}{12! \, 12! \, 10! \, 7! \, 7!},
    \]
    whose $p$-adic order is \(1\), as required.
\end{itemize}
In conclusion, if \(p > 2\), then \(R\) is perfectoid BCM-regular by \cref{Fano-case}.
\end{example}

\bibliographystyle{skalpha}
\bibliography{bibliography.bib}

\end{document}